\documentclass[11pt]{article}
\usepackage{amsmath}
\usepackage{amsfonts}
\usepackage{amssymb}
\usepackage{citesort}
\usepackage{graphicx}
\usepackage{subfigure}%
\newtheorem{theorem}{Theorem}

\newtheorem{corollary}[theorem]{Corollary}

\newtheorem{proposition}[theorem]{Proposition}
\newenvironment{proof}[1][Proof]{\noindent\textbf{#1.} }{\ \rule{0.5em}{0.5em}}

\begin{document}

\title{Inversion of the spherical means transform in corner-like domains by reduction
to the classical Radon transform}
\author{Leonid Kunyansky}
\maketitle

\begin{abstract}
We consider an inverse problem arising in thermo-/photo- acoustic tomography
that amounts to reconstructing a function $f$ from its circular or spherical
means with the centers lying on a given measurement surface. (Equivalently,
these means can be expressed through the solution $u(t,x)$ of the wave
equation with the initial pressure equal to $f$.) An explicit solution of this
inverse problem is obtained in 3D for the surface that is the boundary of an
open octet, and in 2D for the case when the centers of integration circles lie
on two rays starting at the origin and intersecting at the angle equal to
$\pi/N$, $N=2,3,4,...$. Our formulas reconstruct the Radon projections of a
function closely related to $f$, from the values of $u(t,x)$ on the
measurement surface. Then, function $f$ can be found by inverting the Radon transform.
\end{abstract}


\section{Introduction}

Thermoacoustic tomography (TAT)~\cite{KrugerTAT,WangCRC} and photo- or opto-
acoustic tomography (PAT or OAT)~\cite{KrugerPAT,Oraev94,Beard2011} are based
on the thermoacoustic effect: when a material is heated it expands. To perform
measurements, a biological object is illuminated with a short electromagnetic
pulse that heats the tissue and, through the resulting thermoacoustic
expansion, generates an outgoing acoustic wave. Acoustic pressure is then
measured by detectors placed on a surface (completely or partially)
surrounding the object. The source reconstruction inverse problem of TAT/PAT/OAT\ consists in
finding the initial pressure within the object from the measurements.
This pressure is closely related to the blood content in the tissues, making
these imaging modalities particularly effective for cancer detection and
imaging of blood vessels.

Under the assumption that the speed of sound in the tissues is constant, for
certain simple acquisition surfaces a solution of this inverse problem can be
represented by explicit closed-form formulas. Such formulas are similar to the
well-known filtration/backprojection formula in X-ray tomography; they allow
one to compute the initial pressure by evaluating an explicit
integro-differential operator at each node of a reconstruction grid. The
existence and form of explicit inversion formulas are closely related to the
shape of the data acquisition surface. For the simplest case of a planar
surface, explicit formulas have been known for several decades
\cite{Anders,Fawcett,Den}. The authors of \cite{MXW2} found a
filtration/backprojection formula that is valid for a plane, a 3D sphere and
an infinite 3D cylinder. In~\cite{Finch04,Finch07,Kunyansky,Nguyen} several
different inversion formulas were derived for spherical acquisition surfaces
in spaces of various dimensions. In \cite{Kun-cube} explicit reconstruction
formulas were obtained for detectors placed on a surface of a cube (or square,
in 2D) or on surfaces of certain other polygons and polyhedra. Several authors
\cite{Salman,Pala,Nat12,Haltm-ellipse} recently found explicit inversion
formulas for the data measured from surfaces of an ellipse (in 2D) or an
ellipsoid (in 3D) surrounding the object. These formulas can be further
extended by continuity to an elliptic paraboloid or parabolic cylinder
\cite{HP1,HP2}. Certain more complicated polynomial surfaces (including a
paraboloid) were considered in \cite{Pala}, although not all of these surfaces
are attractive from a practical point of view, as they would require
surrounding the object by several layers of detectors. In addition to the
explicit formulas, there exist several reconstruction algorithms (for closed
acquisition surfaces) based on various series expansions
~\cite{Norton1,Norton2,Kun-ser,Kun-cyl,Haltm-series}. These techniques
sometimes lead to very fast implementations (e.g. ~\cite{Kun-ser,Kun-cyl});
however, their efficient numerical implementation may require certain
non-trivial computational skills.

In this paper we derive inversion formulas for acquisition schemes that use
either linear arrays of point- or line- detectors, or two dimensional arrays
of sensors partially surrounding the object. Several experimental setups
motivate our research. For example, in \cite{burg-exac-appro}, linear
detectors were arranged in a corner-like assembly rotated along the object.
Such a scheme requires inversion of series of 2D circular Radon transforms
with centers located on two perpendicular rays. A similar 3D problem sometime
arises when optically scanned planar glass plate is used as an acoustic
detector, as done by researchers from University College London (see, for
example,~\cite{Cox2007,Ell-Cox-apparatus}). The use of a single glass plate
provides only a limited angle view of the object, leading to disappearance of
some edges in the image. To eliminate such artifacts, one needs to repeat
scanning with a detector repositioned to view the object at a different angle,
or to add additional glass plates and/or acoustical mirrors. The methods of
the present paper are applicable in the case of consecutive measurements. The
acquisition scheme with multiple glass detectors may need to be modeled as a
reverberating cavity due to multiple reflections of acoustic waves from the
glass surfaces. The methods of the present paper are not applicable in such
situation; instead, the reader is referred
to~\cite{Cox2007,Ell-Cox-apparatus,Cox,HolKun}.

Recently, linear arrays of point-like piezoelectric detectors started to gain
popularity due to their fast rate of acquisition
(see~\cite{Kroger-linear,Ma,Wang-mirrors}). Similarly to planar arrays, one
such linear assembly yields only limited-angle measurements, and one has to
either use mirrors (as is done in~\cite{Wang-mirrors}), or to arrange several
arrays around the object (reflections from linear arrays can be neglected due
to their smaller cross-section).

In this paper we consider a 3D acquisition scheme where point detectors are
located on the boundary of an infinite octant, and a similar 2D problem with
detectors covering the boundary of an infinite angular sector with an opening
angle $\pi/N$, $N=2,3,4,...$. In both cases it is assumed that reflection of
acoustic waves from the detectors is either absent or negligible. Unlike many
of inversion formulas mentioned above, the formulas we derive are not of a
filtration/backprojection type. Instead, we obtain explicit expressions
allowing one to recover from the measurements the classical Radon projections
of the unknown initial pressure. The latter function then is easily
reconstructed by application of the well-known inverse Radon transform.

\begin{figure}[t]
\begin{center}
\includegraphics[width=4in,height=2.12in]{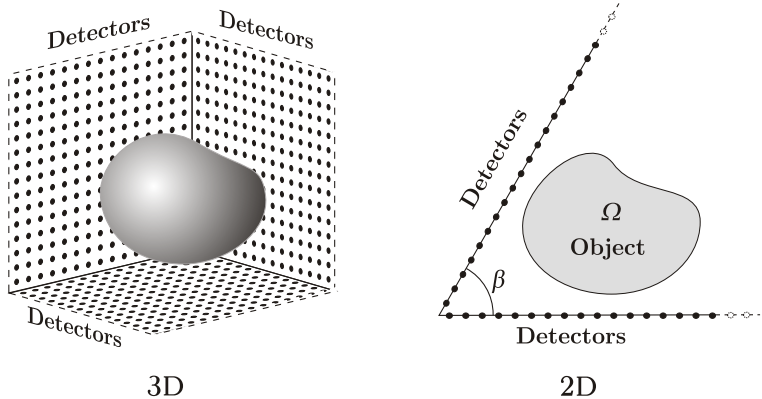}
\end{center}
\caption{Acquisition geometry in 3D and 2D}%
\end{figure}

\section{Preliminaries}

\subsection{Formulation of the problem\label{S:formulation}}

We assume throughout the paper that the attenuation of acoustic waves is
negligible, and that the speed of sound in the tissues is constant (and, thus,
it can be set to 1 without loss of generality). This simplified model is
acceptable in such important applications as breast imaging, and most of the
formulas and algorithms mentioned in the Introduction are based on these
assumptions. Then, the acoustic pressure $p(t,x)$ solves the following initial
problem for the wave equation in the whole space $\mathbb{R}^{d},$ $d=3,$%
\begin{equation}
\left\{
\begin{array}
[c]{c}%
p_{tt}=\Delta p,\text{ }x\in\mathbb{R}^{d},\quad t\in\lbrack0,\infty),\\
p(0,x)=f(x),\\
p_{t}(0,x)=0.
\end{array}
\right.  \label{E:IVP}%
\end{equation}
were the initial pressure of the acoustic wave $f(x)$ is the function we seek.
It is assumed that $f(x)$ is finitely supported within the region of interest
$\Omega,$ and that the point-wise measurements of $p(t,x)$ are done outside
$\Omega.$ A very similar two-dimensional problem arises when linear
integrating detectors are used for measurements (e.g., see \cite{Haltm-series}%
). For example, if a set of linear detectors parallel to vector $e_{3}%
=(0,0,1)$ is utilized, they measure integrals $I(t,x_{1},x_{2})$ of the
pressure along lines parallel to $e_{3}$ and the problem consists in finding
the initial distribution of these integrals $f(x)\equiv I(0,x_{1},x_{2})$.
This problem can also be formulated in the form (\ref{E:IVP}), with $d=2,$
with $p(t,x)$ replaced by $I(t,x_{1},x_{2}).$

Our goal is to reconstruct $f(x)$ from measurements of $p(t,x)$ made at all
points of the measuring surface (or, in 2D, measuring curve)$.$ In 3D, as the
measuring surface we will use the boundary $\partial Q$ of the first Cartesian
octet $Q$ in $\mathbb{R}^{3}$:%
\begin{equation}
Q\equiv\{x\in\mathbb{R}^{3}|x_{1}>0,x_{2}>0,x_{3}>0\}. \label{E:Q-3D}%
\end{equation}
For the 2D problem the data $u(t,x)$ will be assumed known on the boundary of
infinite angular segment $Q$ contained between the ray $\mathbb{R}^{+}%
\equiv\{(c,0)|\ c\geq0\}$ and ray $\mathbb{R}_{\beta}^{+}\equiv\{(c\cos
\beta,c\sin\beta)|\ c\geq0\}$, where angle $\beta$ is to one of the values
$\pi/N,$ $N=2,3,4,...$ $.$ In this case the measuring curve $\partial Q$ is
the union $\mathbb{R}^{+}\cup\mathbb{R}_{\beta}^{+}$. Figure~1 shows
schematically the measuring sets in 2D and 3D. In both cases it is assumed
that the support $\Omega$ of $f$ is properly contained in $Q.$

Thus, our ultimate goal is to find explicit reconstruction formulas expressing
$f(x)$ through measured values $p(t,y),\quad y\in\partial Q,\quad t\in
\lbrack0,\infty).$

\subsection{Outline of the approach\label{S:outline}}

Solution of the initial value problem (\ref{E:IVP}) can be written as a
convolution of $f$ with the time derivative of the retarded (or causal)
Green's function $G^{+}(t,x)$ of the wave equation in free space
\cite{Vladimirov}:
\begin{equation}
p(t,y)=\frac{\partial}{\partial t}\int\limits_{\Omega}f(x)G^{+}(t,y-x)dx.
\label{E:Green-convolution}%
\end{equation}
In 3D the above-mentioned Green's function takes form%
\begin{equation}
G_{(3D)}^{+}(t,x)=\frac{\delta(t-|x|)}{4\pi|x|}, \notag
\end{equation}
where $\delta(t)$ is the Dirac's delta function. In 2D%
\begin{equation}
G_{(2D)}^{+}(t,x)=\left\{
\begin{array}
[c]{cc}%
\frac{1}{2\pi\sqrt{t^{2}-|x|^{2}}} & t>|x|,\\
0 & \text{otherwise}%
\end{array}
\right.  .     \notag
\end{equation}
The measurements $p(t,y)$ are just the values of the convolution
(\ref{E:Green-convolution}) at points $y$ lying on $\partial Q.$ The time
anti-derivative of the data $P(t,y)$ can be easily recovered from $p(t,y)$:%
\begin{equation}
p(t,y)=\frac{\partial}{\partial t}P(t,y),\quad P(0,y)=0,\quad y\in\partial Q. \notag
\end{equation}
Function $P(t,y)$ is related to $\ f(x)$ by the convolution%
\begin{equation}
P(t,y)=\int\limits_{\Omega}f(x)G^{+}(t,y-x)dx,\quad y\in\partial Q. \notag
\end{equation}
For notational
convenience we extend $p(t,y)$ to negative values of $t$ by zero:%
\begin{equation}
p(t,y)=0,\quad t\leq 0.  \label{E:negativetime}
\end{equation}%
(With this convention equation (\ref{E:Green-convolution}) is valid for all $%
t\in \mathbb{R}$.)

Below we will also use the advanced free-space Green's function $G^{-}(t,x)$
that allows one to solve wave equation backwards in time. It can be defined
simply by reversing the sign in front of $t$:%
\begin{equation}
G^{-}(t,x)=G^{+}(-t,x).       \notag
\end{equation}

Consider now an integrable function $\varphi(t,y)$ defined on $\mathbb{R\times
}\partial Q$, finitely supported in the first variable and decaying in $y$
fast enough so that the integral below converges absolutely. One can define a
retarded single layer potential $u_{\varphi}^{+}(t,x)$ with density
$\varphi(t,y)$ by the following equation \cite{Fulks,Tuong}:%
\begin{align*}
u_{\varphi}^{+}(t,x)  &  =\int\limits_{\partial Q}\int\limits_{-\infty
}^{t-|y-x|}\varphi(\tau,y)G^{+}(t-\tau,y-x)d\tau dy\\
&  =\int\limits_{\partial Q}\int\limits_{|y-x|}^{\infty}\varphi(t-s,y)G^{+}%
(s,y-x)dsdy.
\end{align*}
Similarly, an advanced single layer potential is given by the formula%
\begin{align*}
u_{\varphi}^{-}(t,x)  &  =\int\limits_{\partial Q}\int\limits_{-\infty
}^{-|y-x|}\varphi(t-s,y)G^{-}(s,y-x)dsdy\\
&  =\int\limits_{\partial Q}\int\limits_{|y-x|}^{\infty}\varphi(t+s,y)G^{+}%
(s,y-x)dsdy.
\end{align*}
Due to vanishing of $G^{+}(s,y-x)$ when $s<|y-x|,$ one can equivalently write%
\begin{equation}
u_{\varphi}^{\pm}(t,x)=\int\limits_{\partial Q}\int\limits_{\mathbb{R}}%
\varphi(t\mp s,y)G^{+}(s,y-x)dsdy.                      \notag
\end{equation}

Our approach to the inverse problem in hand is based on the observation that
the integral $<f,u_{\varphi}^{-}(0,\cdot)>$ defined as%
\begin{equation}
<f,u_{\varphi}^{-}(0,\cdot)>\equiv\int\limits_{\Omega}f(x)u_{\varphi}%
^{-}(0,x)dx     \notag
\end{equation}
can be easily reconstructed from the data $P(t,y)$:%
\begin{align*}
\int\limits_{\Omega}f(x)u_{\varphi}^{-}(0,x)dx  &  =\int\limits_{\Omega
}f(x)\left[  \int\limits_{\partial Q}\int\limits_{\mathbb{R}}\varphi
(s,y)G^{+}(s,y-x)dsdy\right]  dx\\
&  =\int\limits_{\partial Q}\int\limits_{\mathbb{R}}\varphi(s,y)\left[
\int\limits_{\Omega}f(x)G^{+}(s,y-x)dx\right]  dsdy\\
&  =\int\limits_{\partial Q}\int\limits_{\mathbb{R}}\varphi(s,y)P(s,y)dsdy.
\end{align*}
This procedure can be repeated for\ a set of single layer potentials
$u_{\varphi}^{-}(t,x)$ defined by different densities $\varphi.$ If these
potentials evaluated at $\ t=0$ form some sort of a basis, the above integrals
provide projections of $f$ on the basis elements, making it possible to
reconstruct $f.$

Finding families of basis functions represented by single layer potentials is
not necessarily easy, especially in unbounded domains. In general, certain
solutions of the wave equation can be represented by a combination of a
single- and double- layer potentials. For example, if $u(t,x)$ is a finitely
supported in time solution of the wave equation in a bounded region $Q^{*}$
with a boundary $\partial Q^*$, then $u$ can be represented in the following
form \cite{Fulks,Tuong}%
\begin{align*}
u(t,x)  &  =\int\limits_{\partial Q^*}\int\limits_{\mathbb{R}}\frac{\partial
}{\partial n}u(t\mp s,y)G^{+}(s,y-x)dsdy\\
&  -\int\limits_{\partial Q^*}\int\limits_{\mathbb{R}}u(t\mp s,y)\frac{\partial
}{\partial n}G^{+}(s,y-x)dsdy.
\end{align*}
where second term is a double layer potential with density $u(\tau,y).$ Sign
$\mp$ in the above equation indicates that $u(t,x)$ can be expressed through
its boundary values (including the values of the normal derivative) either in
the past, at the times $\,(t-s),$ or in the future (at the time values $t+s).$

However, for our purposes we need representations expressed only by a single
potential supported on a boundary of a domain, and our domains are not
bounded. In what follows we show that such representations can be obtained for
certain combinations of delta-like plane waves. This will allow us to recover
from the data the integrals of products of $f(x)$ with these functions; such
integrals are related to the values of the Radon transform of $f(x)$. Then,
$f(x)$ can be reconstructed by inverting the Radon transform.

\bigskip

\section{Single layer representations}

In this section we find single layer representations for certain combinations
of plane waves.

\subsection{2D case}

Consider now a Coxeter cross, i.e. a set of $N$ straight lines passing through
the origin, with the angle $\beta=\pi/N$ between any two adjacent lines;
without loss of generality we assume that one of these lines coincides with
$x_{1}$-axis. The boundary of the 2D angular segment $Q$ arising in the 2D
formulation of our inverse problem (see Section \ref{S:formulation}) is a
subset of this set of points. (For a discussion of injectivity of the spherical
mean Radon transform with centers lying on a Coxeter cross we refer the reader
to~\cite{AQ}.)

\subsubsection{An ``odd" operator}

In this section we define an operator that\ maps an arbitrary continuous
function of $x\in\mathbb{R}^{2}$ into a function that is odd with respect to
reflections about all the lines constituting the Coxeter cross.

Define a reflector $\mathfrak{R}$ mapping $x=(x_{1},x_{2})^{T}$ onto
$\mathfrak{R}x=(x_{1},-x_{2})^{T},$ and a linear operator $\Upsilon_{\varphi}$
rotating any vector $x\in\mathbb{R}^{2}$ \ by the angle $\varphi$
counterclockwise$.$ Obviously
\begin{equation}
\Upsilon_{2\beta}^{N}=\Upsilon_{2\beta}^{-N}=I \label{E:rotator-prop1}%
\end{equation}
where $I$ is the identity transformation. Now, let us define an operator $O$
transforming a continuous function $h(x)$ into a continuous function
$h_{O}(x)$ according to the following formula%
\begin{equation}
h_{O}(x)=Oh(x)=\sum_{k=0}^{N-1}\left[  h(\Upsilon_{2\beta}^{k}%
x)-h(\mathfrak{R}\Upsilon_{2\beta}^{k}x)\right]  . \label{E:odd-oper}%
\end{equation}
It follows from (\ref{E:rotator-prop1}) and (\ref{E:odd-oper}) that $h_{O}(x)$
is invariant with respect to rotations by the angle $2\beta$
\begin{equation}
h_{O}(\Upsilon_{2\beta}x)=h_{O}(\Upsilon_{2\beta}^{-1}x)=h_{O}(x),   \notag
\end{equation}
and, in general%
\begin{equation}
h_{O}\left(  \Upsilon_{2\beta}^{k}x\right)  =h_{O}(x),\quad k\in\mathbb{Z}.
\label{E:rot-invar}%
\end{equation}
Further, we notice that due to (\ref{E:rotator-prop1})
\begin{equation}
\sum_{k=0}^{N-1}h(\mathfrak{R}\Upsilon_{2\beta}^{k}x)=\sum_{k=0}%
^{N-1}h(\mathfrak{R}\Upsilon_{2\beta}^{-k}x)=\sum_{k=0}^{N-1}h(\Upsilon
_{2\beta}^{k}\mathfrak{R}x), \label{E:rotator-prop2}%
\end{equation}
where the second equality is verified as follows. Each $x$ can be written as
\ $|x|\Upsilon_{\theta}e_{1}$, where $e_{1}=(1,0)^{T}$ and $\theta$ is some
angle. Then
\begin{equation}
\mathfrak{R}\Upsilon_{2\beta}^{-k}x=\mathfrak{R}\Upsilon_{2\beta}^{-k}\left(
|x|\Upsilon_{\theta}e_{1}\right)  =|x|\mathfrak{R}\Upsilon_{(-2k\beta+\theta
)}e_{1}=|x|\Upsilon_{(2k\beta-\theta)}e_{1}=\Upsilon_{2\beta}^{k}%
\mathfrak{R}x,         \notag
\end{equation}
which proofs the second equality in (\ref{E:rotator-prop2}). Therefore,
operator $O$ (defined by equation (\ref{E:odd-oper})) admits an alternative
representation
\begin{equation}
h_{O}(x)=Oh(x)=\sum_{k=0}^{N-1}\left[  h(\Upsilon_{2\beta}^{k}x)-h(\Upsilon
_{2\beta}^{k}\mathfrak{R}x)\right]  . \label{E:odd-oper1}%
\end{equation}
Since $\mathfrak{R}^{2}=I,$ it is immediately clear from (\ref{E:odd-oper1})
that
\begin{equation}
h_{O}(\mathfrak{R}x)=-h_{O}(x),    \notag
\end{equation}
i.e., $h_{O}(x)$ is an odd function in $x_{2}$. Moreover, due to the rotation
invariance (\ref{E:rot-invar}), $h_{O}(x)$ is odd with respect to any line
that is a rotation of the $x_{1}$ axis by the angle $2\beta k,$ $k\in
\mathbb{\not Z}$, and%
\begin{equation}
h_{O}(\mathfrak{R}\Upsilon_{2\beta}^{k}x)=-h_{O}(x),\quad k\in\mathbb{Z}.
\label{E:oddrefl}%
\end{equation}

One can also show that $h_{O}$ is odd with respect to any line that is a
rotation of $x_{1}$-axis by the angle ($2k+1)\beta,$ $k\in\mathbb{Z}$.
Let's consider two points $c\Upsilon_{\alpha}\Upsilon_{\beta}e_{1}$ and
$c\Upsilon_{-\alpha}\Upsilon_{\beta}e_{1}$ symmetric with respect to the line
$t\Upsilon_{\beta}e_{1}$, $t\in(-\infty,\infty).$ Due to periodicity (equation
(\ref{E:rotator-prop1})), $h_{O}$ admits yet another representation%
\begin{equation}
h_{O}(x)=Oh(x)=\sum_{k=0}^{N-1}\left[  h(\Upsilon_{2\beta}^{k}x)-h(\Upsilon
_{2\beta}^{k+1}\mathfrak{R}x)\right]  .    \notag
\end{equation}
Then%
\begin{align*}
h_{O}(c\Upsilon_{\alpha}\Upsilon_{\beta}e_{1})  &  =\sum_{k=0}^{N-1}\left[
h(c\Upsilon_{2\beta}^{k}x\Upsilon_{\alpha}\Upsilon_{\beta}e_{1})-h(c\Upsilon
_{2\beta}^{k+1}\mathfrak{R}\Upsilon_{\alpha}\Upsilon_{\beta}e_{1})\right] \\
&  =\sum_{k=0}^{N-1}\left[  h(c\Upsilon_{2\beta(k+1)+\alpha}e_{1}%
)-h(c\Upsilon_{2\beta(k+1)-\alpha}e_{1})\right]  =-h_{O}(c\Upsilon_{-\alpha
}\Upsilon_{\beta}e_{1}).
\end{align*}
Therefore, $h_{O}$ is odd with respect to a rotation of $x_{1}$-axis by the
angle $\beta,$ and due to the $2\beta$-periodicity (equation
(\ref{E:rot-invar})), it is odd with respect to a rotation of $x_{1}$-axis by
the angle $(2k+1)\beta.$ Thus, we have proved the following

\begin{proposition}
Operator $O$ defined by equation (\ref{E:odd-oper}) maps a continuous function
into a continuous function that is odd with respect to reflections about each
of the lines constituting the Coxeter cross.
\end{proposition}

\begin{corollary}
If $h_{O}(x)=Oh(x)$, then $h_{O}(x)$ vanishes on each of the lines
constituting the Coxeter cross, and, in particular, on the boundary $\partial
Q$ of segment $Q.$
\end{corollary}

\begin{proposition}
If support of $h(x)$ is contained within $Q,$ then restriction of $h_{O}$ to
$Q$ coincides with $h,$ i.e.%
\begin{equation}
h(x)=h_{O}(x),\quad\forall x\in Q.     \notag
\end{equation}

\end{proposition}

\begin{proof}
If support of $h(x)$ is contained within $Q$, then each term in
(\ref{E:odd-oper}) that is not $h(x)$ itself, is supported within one of
rotated and/or reflected versions of segment $Q$ not intersecting $Q$ itself.
\end{proof}

\begin{proposition}
Suppose $g(x)$ and $h(x)$ are continuous functions, one of them is compactly
supported, with $g_{O}(x)=Og(x)$ and $h_{O}(x)=Oh(x).$ Then%
\begin{equation}
\int\limits_{\mathbb{R}^{2}}g_{O}(x)h(x)dx=\int\limits_{\mathbb{R}^{2}}%
h_{O}(x)g(x)dx.   \notag
\end{equation}
\label{T:oddcommute}
\end{proposition}

\begin{proof}
By using the definition of operator $O$ and interchanging the order of
summation and integration one obtains%
\begin{align*}
\int\limits_{\mathbb{R}^{2}}g_{O}(x)h(x)dx  &  =\sum_{k=0}^{N-1}%
\int\limits_{\mathbb{R}^{2}}g(\Upsilon_{2\beta}^{k}x)h(x)dx-\sum_{k=0}%
^{N-1}\int\limits_{\mathbb{R}^{2}}g(\mathfrak{R}\Upsilon_{2\beta}%
^{k}x)h(x)dx\\
&  =\sum_{k=0}^{N-1}\int\limits_{\mathbb{R}^{2}}g(x)h(\Upsilon_{2\beta}%
^{-k}x)dx-\sum_{k=0}^{N-1}\int\limits_{\mathbb{R}^{2}}g(x)h(\mathfrak{R}%
\Upsilon_{2\beta}^{-k}x)dx\\
&  =\int\limits_{\mathbb{R}^{2}}g(x)\left(  \sum_{k=0}^{N-1}\left[
h(\Upsilon_{2\beta}^{-k}x)-h(\mathfrak{R}\Upsilon_{2\beta}^{-k}x)\right]
\right)  dx=\\
&  =\int\limits_{\mathbb{R}^{2}}g(x)\left(  \sum_{k=0}^{N-1}\left[
h(\Upsilon_{2\beta}^{k}x)-h(\mathfrak{R}\Upsilon_{2\beta}^{k}x)\right]
\right)  dx=\int\limits_{\mathbb{R}^{2}}h_{O}(x)g(x)dx,
\end{align*}
where (\ref{E:rotator-prop1}) was used on the last line.
\end{proof}

Finally, the following statement follows directly from (\ref{E:rot-invar}) and
(\ref{E:oddrefl}).

\begin{proposition}
\label{T:additionalsym2D}If $g(x)$ and $h(x)$ are continuous functions and
$g(x)$ is compactly supported then%
\begin{align*}
\int\limits_{\mathbb{R}^{2}}g(x)h_{O}\left(  \Upsilon_{2\beta}^{k}x\right)
dx  &  =\int\limits_{\mathbb{R}^{2}}g(x)h_{O}(x)dx,\\
\int\limits_{\mathbb{R}^{2}}g(x)h_{O}\left(  \Upsilon_{2\beta}^{k}%
\mathfrak{R}x\right)  dx  &  =-\int\limits_{\mathbb{R}^{2}}g(x)h_{O}\left(
x\right)  dx,\quad k\in\mathbb{Z}.
\end{align*}

\end{proposition}

\bigskip

\bigskip

\subsection{3D case}

In the 3D version of our problem, region $Q$ is the first Cartesian octet $Q$
defined by equation (\ref{E:Q-3D}).

\subsubsection{An ``odd" operator}

In this section we define an operator that\ maps an arbitrary continuous
function of 3D variable into function that is odd with respect to reflections
about any of the three coordinate planes.

Let us define operators $\mathfrak{R}j,$ $j=1,2,3,$ reflecting any
$x=(x_{1},x_{2},x_{3})^{T}$ with respect to the coordinate plane with the
normal $e_{j}$:%
\begin{align}
\mathfrak{R}_{1}x  &  =(-x_{1},x_{2},x_{3})^{T},\nonumber\\
\mathfrak{R}_{2}x  &  =(x_{1},-x_{2},x_{3})^{T},\label{E:reflector3D}\\
\mathfrak{R}_{3}x  &  =(x_{1},x_{2},-x_{3})^{T}.\nonumber
\end{align}
It follows from the above definition that these reflectors commute.

We can now define an operator $O$\ mapping continuous function $h(x)$ into
function $h_{O}(x)$ odd with respect to reflections about each of the three
coordinate planes$:$%
\begin{align}
h_{O}(x)  &  =Oh(x)\equiv h(x)-h(\mathfrak{R}_{1}x)-h(\mathfrak{R}%
_{2}x)-h(\mathfrak{R}_{3}x)\nonumber\\
&  +h(\mathfrak{R}_{2}\mathfrak{R}_{1}x)+h(\mathfrak{R}_{3}\mathfrak{R}%
_{1}x)+h(\mathfrak{R}_{3}\mathfrak{R}_{2}x)-h(\mathfrak{R}_{3}\mathfrak{R}%
_{2}\mathfrak{R}_{1}x). \label{E:odd-oper3D}%
\end{align}
Since operators $\mathfrak{R}j$ commute and are self-invertible, verifying
that operator $O$ have the above-mentioned property is straightforward.

\begin{proposition}
Operator $O$ defined by equation (\ref{E:odd-oper3D}) maps a continuous
function into a continuous function that is odd with respect to reflections
about each of the three coordinate planes.
\end{proposition}

Similarly, one proves

\begin{proposition}
\label{T:3D-refl}If $h(x)$ is bounded and $h_{O}(x)=Oh(x)$ then%
\begin{equation}
h_{O}\left(  \mathfrak{R}_{i}x\right)  =-h_{O}(x),\quad h_{O}\left(
\mathfrak{R}_{i}\mathfrak{R}_{j}x\right)  =h_{O}(x),\quad h_{O}\left(
\mathfrak{R}_{i}\mathfrak{R}_{j}\mathfrak{R}_{k}x\right)  =-h_{O}(x),\quad
i,j,k\in\{1,2,3\}.         \notag
\end{equation}

\end{proposition}

\begin{corollary}
If $h_{O}(x)=Oh(x)$, then $h_{O}(x)$ vanishes on each of the coordinate
planes, in particular, on the boundary $\partial Q$ of the first octet $Q.$
\end{corollary}

\begin{proposition}
If support of $h(x)$ is contained within $Q,$ then restriction of $h_{O}$ to
$Q$ coincides with $h,$ i.e.%
\begin{equation}
h(x)=h_{O}(x),\quad\forall x\in Q.   \notag
\end{equation}

\end{proposition}

\begin{proposition}
Suppose $g(x)$ and $h(x)$ are continuous functions, one of them is compactly
supported, with $g_{O}(x)=Og(x)$ and $h_{O}(x)=Oh(x).$ Then%
\begin{equation}
\int\limits_{\mathbb{R}^{3}}g_{O}(x)h(x)dx=\int\limits_{\mathbb{R}^{3}}%
h_{O}(x)g(x)dx.        \notag
\end{equation}

\end{proposition}

The proof of the last proposition is similar to that of Proposition
\ref{T:oddcommute}.

As a direct consequence of Proposition \ref{T:3D-refl},\ one obtains \bigskip

\begin{proposition}
\label{T:additionalsym3D}If $g(x)$ and $h(x)$ are continuous functions, and
$g(x)$ is compactly supported, then%
\begin{align*}
\int\limits_{\mathbb{R}^{3}}g(x)h_{O}\left(  \mathfrak{R}_{i}\right)  dx  &
=-\int\limits_{\mathbb{R}^{3}}g(x)h_{O}(x)dx,\\
\int\limits_{\mathbb{R}^{3}}g(x)h_{O}\left(  \mathfrak{R}_{i}\mathfrak{R}%
_{j}x\right)   &  =\int\limits_{\mathbb{R}^{3}}g(x)h_{O}(x)dx,\\
\int\limits_{\mathbb{R}^{3}}g(x)h_{O}\left(  \mathfrak{R}_{i}\mathfrak{R}%
_{j}\mathfrak{R}_{k}x\right)   &  =-\int\limits_{\mathbb{R}^{3}}%
g(x)h_{O}(x)dx,\quad i,j,k\in\{1,2,3\}.
\end{align*}

\end{proposition}

\bigskip

\subsection{Waves represented by single layer potentials}

\subsubsection{Odd combinations of waves}

Consider a non-negative, infinitely smooth, even function $\eta\left(
t\right)  $ \ defined on $\mathbb{R}$, compactly supported with all its
derivatives on $(-1,1)$, and such that $\int_{-1}^1 \eta dt = 1$. Define scaled function $\eta_{\varepsilon}(t)$ by the
formula%
\begin{equation}
\eta_{\varepsilon}(t)=\frac{1}{\varepsilon}\eta\left(  \frac{1}{\varepsilon
}t\right)  .     \notag
\end{equation}
Family of functions $\eta_{\varepsilon}(t)$ is delta-approximating, in the sense of distributions:%
\begin{equation}
\lim_{\varepsilon\rightarrow0}\eta_{\varepsilon}(t)=\delta(t). \label{E:delta}%
\end{equation}
Now, let us consider a plane wave $u_{\varepsilon}(t,x)$ defined by the
equation
\begin{equation}
u_{\varepsilon,\omega}(t,x)=\eta_{\varepsilon}\left(  t-x\cdot\omega\right), \notag
\end{equation}
where $\omega$ is a unit vector defining the direction of propagation of such
a wave. It is easy to check that $u_{\varepsilon}(t,x)$ is a solution of the
wave equation in the whole space $\mathbb{R}^{d}$:%
\begin{equation}
\Delta u_{\varepsilon,\omega}=\frac{\partial^{2}}{\partial t^{2}%
}u_{\varepsilon,\omega},\qquad x\in\mathbb{R}^{d},\quad t\in\mathbb{R},   \notag
\end{equation}
and that, in the limit $\varepsilon\rightarrow0,$ $u_{\varepsilon,\omega}$
approximates a delta-wave%
\begin{equation}
\lim_{\varepsilon\rightarrow0}u_{\varepsilon,\omega}(t,x)=\delta
(t-x\cdot\omega).   \notag
\end{equation}

Let us now use the operator $O$ introduced in the previous sections and define
the following combination of plane waves:%
\begin{equation}
U_{\varepsilon,\omega}(t,x)=Ou_{\varepsilon,\omega}(t,x). \label{E:capitalU}%
\end{equation}
To simplify the notation, within this section we will suppress the subscripts
and denote this function by $U(t,x).$ Function $U(t,x)$ consists of the plane
wave $u_{\varepsilon,\omega}(t,x)$ together with its reflections and (in 2D
case) rotations, selected in such a way that it is odd with respect to a set
of lines (or planes) containing the boundary of $\partial Q.$ Thus,
\begin{equation}
U(t,x)=0,\qquad\forall x\in\partial Q,\quad\forall t\in\mathbb{R}.    \notag
\end{equation}

\begin{figure}[t]
\begin{center}
\subfigure[Wave $u_{\varepsilon,\omega}(t,x)$]{
\includegraphics[width=1.6in,height=1.6in]{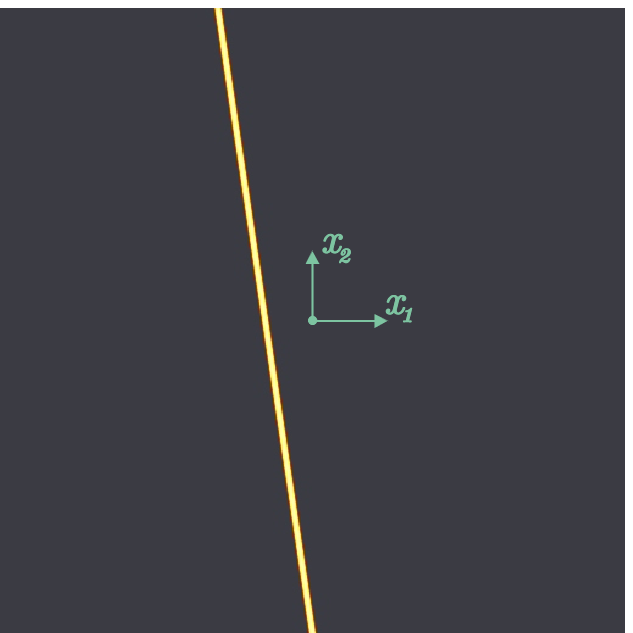}}
\subfigure[$U(t,x)$, case $N=2$]{
\includegraphics[width=1.6in,height=1.6in]{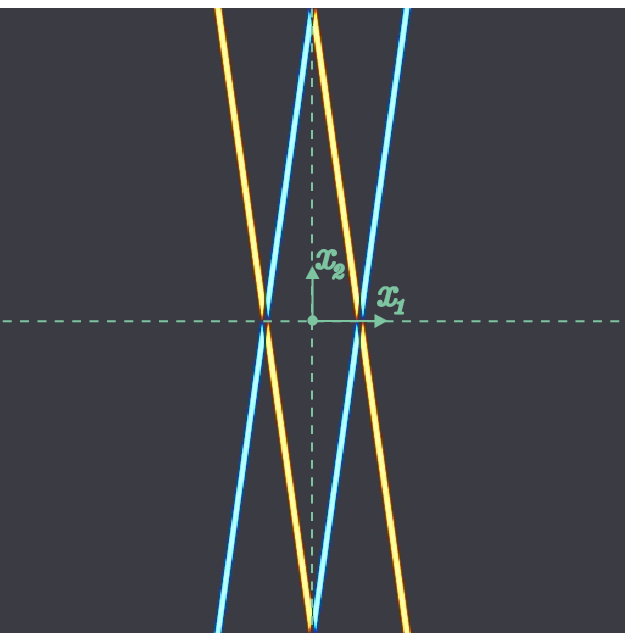}}
\subfigure[$U(t,x)$, case $N=3$]{
\includegraphics[width=1.6in,height=1.6in]{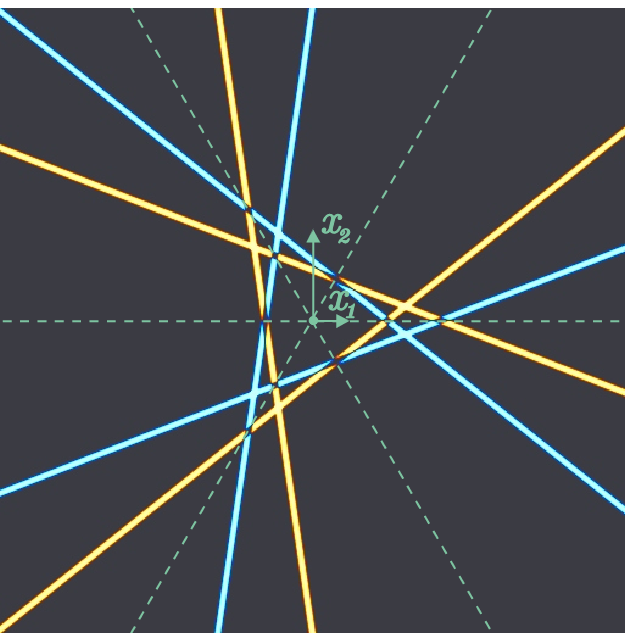}
\includegraphics[width=0.32in,height=1.6in]{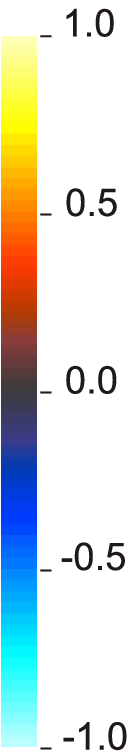}}
\end{center}
\caption{Constructing $U(t,x)$ from $u_{\varepsilon,\omega}(t,x)$ in the cases
$N=2$ and $N=3$. The dashed line shows the lines of the Coxeter cross;
$U(t,x)$ is odd with respect to these lines}%
\label{F:crosses}%
\end{figure}

From now on, let us assume that vector $\omega$ is pointing strictly inside $Q$.
In the 2D case the operator $O$ is defined by equation~(\ref{E:odd-oper}).

Apply it to $u_{\varepsilon,\omega}(t,x)$:
\begin{align}
U(t,x)  &  =Ou_{\varepsilon,\omega}(t,x)=\nonumber\\
&  =\sum_{k=0}^{N-1}\left[  \eta_{\varepsilon}\left(  t-\Upsilon_{2\beta}%
^{k}x\cdot\omega\right)  -\eta_{\varepsilon}\left(  t-\mathfrak{R}%
\Upsilon_{2\beta}^{k}x\cdot\omega\right)  \right] \nonumber\\
&  =\sum_{k=0}^{N-1}\left[  \eta_{\varepsilon}\left(  t-x\cdot\Upsilon
_{-2\beta}^{k}\omega\right)  -\eta_{\varepsilon}\left(  t-x\cdot
\mathfrak{R}\Upsilon_{-2\beta}^{k}\omega\right)  \right]  .
\label{E:U-details}%
\end{align}
It follows that $U(t,x)$ has the following form%
\begin{align}
U(t,x)  &  =\eta_{\varepsilon}\left(  t-x\cdot\omega\right)  +U^{\ast
}(t,x),\label{E:Ustar-sum}\\
U^{\ast}(t,x)  &  \equiv\sum_{j=1}^{2N-1}\sigma_{j}\eta_{\varepsilon}\left(
t-x\cdot\omega^{(j)}\right)  , \label{E:Ustar}%
\end{align}
where $\sigma_{j}$ equal $1$ or $-1,$ and unit vectors $\omega^{(j)},$
$j=1,..,2N-1,$ are rotated and/or reflected versions of $\omega.$ By
construction of operators $\mathfrak{R}$ and $\Upsilon_{-2\beta}^{k}$ entering
equation (\ref{E:U-details}), each of vectors $\omega^{(j)}$ lies in its own
sector of the Coxeter cross. Figure~\ref{F:crosses} shows two examples of
constructing $U(t,x)$ from $u_{\varepsilon,\omega}(t,x)$, for the cases $N=2$
and $N=3$.

In the 3D case, function $U(t,x)$ has a similar form, given by
equations~(\ref{E:Ustar-sum}),~(\ref{E:Ustar}) with $N=4;$ each of the vectors
$\omega^{(j)}$ in this case is pointing inside its own octet of the Cartesian grid.

\subsubsection{Representing odd waves by single layer potentials\label{S:repsinglay}}

We would like to represent $U(t,x)$ in $Q$ by an advanced single layer
potential supported on $\partial Q.$

Let us first consider a bounded subset $Q^{\ast}$ of $Q$ defined as follows.
Let $B(0,R)$ be the open ball of radius $R$ centered at the origin. $R$ will
be used as a large parameter. Define $Q^{\ast}\equiv Q\cap B(0,R).$ Boundary
$\partial Q^{\ast}$ of $Q^{\ast}$ consists of two parts:%
\begin{align*}
\partial Q^{\ast}  &  =\partial Q_{1}^{\ast}\cup\partial Q_{2}^{\ast},\\
\partial Q_{1}^{\ast}  &  =\partial Q\cap B(0,R)\\
\partial Q_{2}^{\ast}  &  =Q\cap\partial B(0,R).
\end{align*}
Since $\eta_{\varepsilon}$ is compactly supported on $(-\varepsilon, \varepsilon)$, function $U(t,x)$ identically vanishes in $B(0,R)$ for all
$t$ such that $t>t_{0}\equiv R+\varepsilon.$ Therefore, for all $x$ in
$Q^{\ast}$%
\begin{align}
U(t,x)  &  =\int\limits_{\partial Q^{\ast}}\int\limits_{|y-x|}^{\infty}\left[
G^{+}(s,y-x)\frac{\partial}{\partial n}U(t+s,y)dsdy-U(t+s,y)\frac{\partial
}{\partial n}G^{+}(s,y-x)\right]  dsdy\nonumber\\
&  =\int\limits_{\partial Q_{1}^{\ast}}\int\limits_{|y-x|}^{\infty}%
G^{+}(s,y-x)\frac{\partial}{\partial n}U(t+s,y)dsdy\nonumber\\
&  +\int\limits_{\partial Q_{2}^{\ast}}\int\limits_{|y-x|}^{\infty}\left[
G^{+}(s,y-x)\frac{\partial}{\partial n}U(t+s,y)dsdy-U(t+s,y)\frac{\partial
}{\partial n}G^{+}(s,y-x)\right]  dsdy, \label{E:repres1}%
\end{align}
where second equality is due to vanishing of $U(t,y)$ on $\partial Q$ and,
thus, on $\partial Q_{1}^{\ast}.$ Let us denote by $I$ the double integral on
the last line of the above equation. Then, due to~(\ref{E:Ustar-sum})%
\begin{align*}
I  &  =\int\limits_{\partial Q_{2}^{\ast}}\int\limits_{|y-x|}^{\infty}\left[
G^{+}(s,y-x)\frac{\partial}{\partial n}\eta_{\varepsilon}\left(
t+s-y\cdot\omega\right)  dsdy-\eta_{\varepsilon}\left(  t+s-y\cdot
\omega\right)  \frac{\partial}{\partial n}G^{+}(s,y-x)\right]  dsdy\\
&  +\int\limits_{\partial Q_{2}^{\ast}}\int\limits_{|y-x|}^{\infty}\left[
G^{+}(s,y-x)\frac{\partial}{\partial n}U^{\ast}(t+s,y)dsdy-U^{\ast
}(t+s,y)\frac{\partial}{\partial n}G^{+}(s,y-x)\right]  dsdy.
\end{align*}
Let us show that the last double integral in the above equation vanishes if
$R$ is chosen sufficiently large. \ Let us first consider only values of $x$
lying within region $A=Q\cap B(0,r_{0}),$ and values of $t$ such that $|t|\leq
r_{0}.$ We notice that for fixed $x$ and $y,$ integration in time variable is
done only over values $t+s,$ with $s>|x-y|$. Consider values of an arbitrary
term in $U^{\ast}$ (say, with subscript $j$ as defined by equation
(\ref{E:Ustar})) at the point $y$ and at time $t+s.$ It has form
$\eta_{\varepsilon}\left(  t+s-y\cdot\omega^{(j)}\right)  $. Function
$\eta_{\varepsilon}\left(  t+s-y\cdot\omega^{(j)}\right)  $ and its normal
derivative vanish outside of the interval%
\begin{equation}
-\varepsilon+y\cdot\omega^{(j)}<t+s<\varepsilon+y\cdot\omega^{(j)}.       \notag
\end{equation}
On the other hand, for $x\in A,$ the smallest value of $|y-x|$ is $R-r_{0},$
so that integration is done over the interval
\begin{equation}
s\geq R-r_{0}       \notag
\end{equation}
or
\begin{equation}
t+s\geq R-2r_{0}.     \notag
\end{equation}
It is sufficient to show that there is $R$ such that%
\begin{equation}
\varepsilon+y\cdot\omega^{(j)}<R-2r_{0} \label{E:inequ1}%
\end{equation}
for all $y\in\partial Q_{2}^{\ast}.$ Any such $y$ can be represented as
$y=\zeta R$, where $\zeta$ is a unit vector lying strictly within sector $Q.$
Since each $\omega^{(j)},$ $j=1,...,2N-1$ is fixed and lies strictly outside
$Q,$ there is a uniform (in $y$ and $j$) upper bound $\alpha$ on the dot product
$\zeta\cdot\omega^{(j)}$:%
\begin{equation}
\zeta\cdot\omega^{(j)}<\alpha<1. \label{E:inequ5}%
\end{equation}
Therefore, for the left hand side of (\ref{E:inequ1}) we obtain%
\begin{equation}
\varepsilon+y\cdot\omega^{(j)}<\varepsilon+R\alpha.         \notag
\end{equation}
It is easy to check that if $R$ is chosen so that%
\begin{equation}
R>\frac{\varepsilon+2r_{0}}{(1-\alpha)}, \label{E:inequ3}%
\end{equation}
then inequality (\ref{E:inequ1}) is satisfied for each $y\in\partial
Q_{2}^{\ast},$ $x\in A,$ and $t$ such that $|t|\leq r_{0}.$ Therefore, for
such $x$ and $t$ equation (\ref{E:repres1}) simplifies to
\begin{align}
U(t,x)  &  =\int\limits_{\partial Q_{1}^{\ast}}\int\limits_{|y-x|}^{\infty
}G^{+}(s,y-x)\frac{\partial}{\partial n}U(t+s,y)dsdy+\int\limits_{\partial
Q_{2}^{\ast}}\int\limits_{|y-x|}^{\infty}\left[  G^{+}(s,y-x)\frac{\partial
}{\partial n}\eta_{\varepsilon}\left(  t+s-y\cdot\omega\right)  \right.
\nonumber\\
&  -\left.  \eta_{\varepsilon}\left(  t+s-y\cdot\omega\right)  \frac{\partial
}{\partial n}G^{+}(s,y-x)\right]  dsdy. \label{E:interm}%
\end{align}
Let us now show that the second double integral in (\ref{E:interm}) equals
$\eta_{\varepsilon}\left(  t+s-y\cdot\omega\right)  $ for $|x|<r_{0}$ and
$|t|<r_{0}.$ Consider again ball $B(0,R)$ with the boundary $S(0,R)$. Within
this ball $\eta_{\varepsilon}\left(  t+s-y\cdot\omega\right)  $ can be
represented by the integral
\begin{equation}
\eta_{\varepsilon}=\int\limits_{S(0,R)}\int\limits_{|y-x|}^{\infty}\left[
G^{+}(s,y-x)\frac{\partial}{\partial n}\eta_{\varepsilon}\left(
t+s-y\cdot\omega\right)  -\eta_{\varepsilon}\left(  t+s-y\cdot\omega\right)
\frac{\partial}{\partial n}G^{+}(s,y-x)\right]  dsdy. \label{E:onewave}%
\end{equation}
Considerations similar to the above, show that if $R$ is chosen so that%
\begin{equation}
R>\frac{\varepsilon+2r_{0}}{(1-\beta)}, \label{E:inequ2}%
\end{equation}
with some $\beta<1,$ then the inner integral in (\ref{E:onewave}) vanishes for
all values $y=R \zeta$ such that%
\begin{equation}
\zeta\cdot\omega<\beta.              \notag
\end{equation}
If one selects $\beta$ close enough to $1$ to guarantee that all $y \in \partial Q_{2}^{\ast}$
satisfy $\zeta\cdot\omega<\beta$, and if $R$
satisfies (\ref{E:inequ2}), then
\begin{equation}
\eta_{\varepsilon}=\int\limits_{\partial Q_{2}^{\ast}}\int\limits_{|y-x|}%
^{\infty}\left[  G^{+}(s,y-x)\frac{\partial}{\partial n}\eta_{\varepsilon
}\left(  t+s-y\cdot\omega\right)  -\eta_{\varepsilon}\left(  t+s-y\cdot
\omega\right)  \frac{\partial}{\partial n}G^{+}(s,y-x)\right]  dsdy.
\label{E:onewave2}%
\end{equation}
Now, by comparing (\ref{E:interm}) and (\ref{E:onewave2}) one obtains%
\begin{equation}
U(t,x)-\eta_{\varepsilon}\left(  t-x\cdot\omega\right)  =\int\limits_{\partial
Q_{1}^{\ast}}\int\limits_{|y-x|}^{\infty}G^{+}(s,y-x)\frac{\partial}{\partial
n}U(t+s,y)dsdy,\quad|x|\leq r_{0},\quad|t|\leq r_{0}, \label{E:onewave3}%
\end{equation}
valid if $R$ satisfies both (\ref{E:inequ3}) and (\ref{E:inequ2}). Finally, we
notice that for $t\in\lbrack-r_{0},-\varepsilon]$ support of the wave
$\eta_{\varepsilon}\left(  t-x\cdot\omega\right)  $ does not intersect $Q,$
and therefore this term in (\ref{E:onewave3}) can be dropped:%
\begin{equation}
U(t,x)=\int\limits_{\partial Q_{1}^{\ast}}\int\limits_{|y-x|}^{\infty}%
G^{+}(s,y-x)\frac{\partial}{\partial n}U(t+s,y)dsdy,\quad|x|\leq r_{0},\quad
t\in\lbrack-r_{0},-\varepsilon], \quad  x \in Q.           \notag
\end{equation}
Since for a fixed $\omega$ this representation is valid for any sufficiently
large $R,$ one can take a limit $R\rightarrow\infty$ and replace the
integration over $\partial Q_{1}^{\ast}$ by integration over $\partial Q.$
Moreover, since $r_{0}$ is arbitrary, this representation is valid within any
bounded region $\Omega$ properly contained in $Q,$ for $t<-\varepsilon$:%
\begin{equation}
U(t,x)=\int\limits_{\partial Q}\int\limits_{|y-x|}^{\infty}G^{+}%
(s,y-x)\frac{\partial}{\partial n}U(t+s,y)dsdy,\quad t<-\varepsilon, \quad  x \in Q.
\label{E:single-layer-rep}%
\end{equation}
Thus, we have proven


\begin{proposition}
Let $U(t,x)$ be a combination of plane waves defined by equations (\ref%
{E:Ustar-sum}) and (\ref{E:Ustar}), with infinitely smooth $\eta
_{\varepsilon }$ supported on the interval $[-\varepsilon ,\varepsilon ]$
(where $\varepsilon >0$ is arbitrary) and vanishing with its derivatives at
the endpoints of the interval. Then $U(t,x)$ can be represented within any
bounded region $\Omega $ properly contained in $Q$ by an advanced single
layer potential~(\ref{E:single-layer-rep}); the representation is valid for
all $t<-\varepsilon .$
\end{proposition}

%
%
%

Let us make one more observation on the single layer representation (\ref%
{E:single-layer-rep}). Spatial integration in this formula can be restricted
to a bounded subset of $\partial Q,$ per the following

\begin{proposition}
\label{T:boundedsupport}Let function $U(t,x)$ be defined as in the previous
proposition, with fixed function $\eta_{\varepsilon }$, and fixed values  $\varepsilon >0$ and $%
\omega \in Q.$ For arbitrary value $r_{0}>\varepsilon>0$ consider a bounded subset
$Q^{\ast} \equiv Q\cap B(0,r_{0})$ of $Q$. Then, there exists
value $R=R(\omega ,\varepsilon ,r_{0})$ such that the following
representation of $U(t,x)$ holds within $Q^*$:%
\begin{equation}
U(t,x)=\int\limits_{\partial Q\cap B(0,R))} \
\int\limits_{|y-x|}^{\infty }G^{+}(s,y-x)\frac{\partial }{\partial n}%
U(t+s,y)dsdy,\ -r_{0}<t<-\varepsilon ,\quad x\in Q^{\ast }.
\label{E:single-layer-rep1}
\end{equation}
\end{proposition}

\begin{proof}
The proof of this statement is quite similar to the proof of the previous
proposition. Here is a brief sketch of it. All the points $y$ on the
integration surface $\partial Q$ can be represented in the form $y=\zeta|y|.$
Since vector $\omega $ and its reflections/rotations $\omega ^{(j)}$ lie
correspondingly strictly within $Q$ and within  reflections/rotations of $Q$%
, there is an upper bound $\alpha (\omega ),$ such that an inequality similar
to (\ref{E:inequ5}) is satisfied, with $R=|y|$ and $\alpha =\alpha (\omega ).
$ Therefore, for values of $y$ such that
\begin{equation}
|y|>R(\omega ,\varepsilon ,r_{0})\equiv \frac{\varepsilon +2r_{0}}{(1-\alpha
(\omega ))},  \label{E:finitesupport}
\end{equation}%
the inner integral in (\ref{E:single-layer-rep}) vanishes, since support of $%
\frac{\partial }{\partial n}U(t+s,y)$ is disjoint from the integration
interval. Hence, equation (\ref{E:single-layer-rep1}) holds.
\end{proof}

\section{Reconstruction of Radon projections}

The goal of this section is to recover the Radon projections of $f$ from
measurements $p(t,y).$ More precisely, we will be able to reconstruct
projections of $Of$; this, however, is enough to reconstruct $f$ by inverting
the Radon transform.

The latter transform, $\mathcal{R}$, of an arbitrary continuous compactly
supported function $h(x)$ can be defined as follows \cite{Natterer}:
\begin{equation}
\left(  \mathcal{R}h\right)  (t,\omega)=\int_{\mathbb{R}^{d}}h(x)\delta
(t-x\cdot\omega)dx,\quad d=2,3, \label{E:Radon}%
\end{equation}
where unit vector $\omega$ is lying on a unit circle (in 2D) or a sphere
$\mathbb{S}^{2}$ (in 3D). For a fixed values of $\omega$ and $t$ the Radon
transform equals to an integral of a function over a hyperplane normal to
$\omega$ and lying at a distance $|t|$ from the origin. The range of
variable $t$ is chosen so that all hyperplanes intersecting the support of $h$
are accounted for. Inversion formulas allowing one to reconstruct $h$ from
known values of $\mathcal{R}h$ are well known, along with a number of
efficient reconstruction algorithms (e.g., \cite{Natterer,Kuchment}).

Following the idea presented in Section \ref{S:outline}, we utilize the single
layer representation (\ref{E:single-layer-rep}) and find the integral of
$f(x)$ with the combination of plane waves $U_{\varepsilon,\omega}(t,x)$
defined by equation (\ref{E:capitalU}):%

\begin{align}
\int\limits_{\Omega}f(x)U_{\varepsilon,\omega}(t,x)dx  &  =\int\limits_{\Omega
}f(x)\left[  \int\limits_{\partial Q}\int\limits_{|y-x|}^{\infty}%
G^{+}(s,y-x)\frac{\partial}{\partial n}U_{\varepsilon,\omega}%
(t+s,y)dsdy\right]  dx\nonumber\\
&  =\int\limits_{\partial Q}\int\limits_0^\infty\frac{\partial}{\partial
n}U_{\varepsilon,\omega}(t+s,y)\left[  \int\limits_{\Omega}f(x)G^{+}%
(s,y-x)dx\right]  dsdy\nonumber\\
&  =\int\limits_{\partial Q}\int\limits_0^\infty P(s,y)\frac{\partial
}{\partial n}U_{\varepsilon,\omega}(t+s,y)dsdy. \label{E:inner-pr}%
\end{align}
Let us modify formulas (\ref{E:Ustar}) and (\ref{E:Ustar-sum}) and write%
\begin{equation}
U_{\varepsilon,\omega}(t,x)=\sum_{j=0}^{2N-1}\sigma_{j}\eta_{\varepsilon
}\left(  t-x\cdot\omega^{(j)}\right)  ,  \notag
\end{equation}
with $\sigma_{0}=1$ and $\omega^{(0)}=\omega.$ Now the normal derivative of
$U_{\varepsilon,\omega}$ can be computed explicitly:%
\begin{align*}
\frac{\partial}{\partial n(y)}U_{\varepsilon,\omega}(t+s,y)  &  =n(y)\cdot
\nabla U_{\varepsilon,\omega}(t+s,y)=-\sum_{j=0}^{2N-1}\sigma_{j}\left(
n(y)\cdot\omega^{(j)}\right)  \eta_{\varepsilon}^{\prime}\left(
t+s-y\cdot\omega^{(j)}\right) \\
&  =-\frac{\partial}{\partial s}\sum_{j=0}^{2N-1}\sigma_{j}\left(
n(y)\cdot\omega^{(j)}\right)  \eta_{\varepsilon}\left(  t+s-y\cdot\omega
^{(j)}\right)  .
\end{align*}
\newline Substitute the above expression in (\ref{E:inner-pr}) and integrate
by parts in $s$:%
\begin{equation}
\int\limits_{\Omega}f(x)U_{\varepsilon,\omega}(t,x)dx=\int\limits_{\partial
Q}\int\limits_0^\infty p(s,y)\left[  \sum_{j=0}^{2N-1}\sigma_{j}\left(
n(y)\cdot\omega^{(j)}\right)  \eta_{\varepsilon}\left(  t+s-y\cdot\omega
^{(j)}\right)  \right]  dsdy.  \notag
\end{equation}
By taking the limit $\varepsilon\rightarrow0$ and recalling
(\ref{E:negativetime}) and (\ref{E:delta}) we obtain
\begin{equation}
\lim_{\varepsilon\rightarrow0}\int\limits_{\Omega}f(x)U_{\varepsilon,\omega
}(t,x)dx=\int\limits_{\partial Q}\left[  \sum_{j=0}^{2N-1}\sigma_{j}\left(
n(y)\cdot\omega^{(j)}\right)  p(y\cdot\omega^{(j)}-t,y)\right]  dy.
\notag \end{equation}
The left hand side of the above equation can be modified using Proposition
\ref{T:oddcommute}:%
\begin{align*}
\lim_{\varepsilon\rightarrow0}\int\limits_{\Omega}f(x)U_{\varepsilon,\omega
}(t,x)dx  &  =\lim_{\varepsilon\rightarrow0}\int\limits_{\Omega}%
f(x)Ou_{\varepsilon}(t,x)dx=\lim_{\varepsilon\rightarrow0}\int\limits_{\Omega
}u_{\varepsilon}(t,x)Of(x)dx\\
&  =\lim_{\varepsilon\rightarrow0}\int\limits_{\Omega}u_{\varepsilon
}(t,x)Of(x)dx=\lim_{\varepsilon\rightarrow0}\int\limits_{\Omega}%
Of(x)\eta_{\varepsilon}\left(  t-x\cdot\omega\right)  dx\\
&  =\int\limits_{\Omega}Of(x)\delta\left(  t-x\cdot\omega\right)  dx=\left(
\mathcal{R}\left[  f_{O}\right]  \right)  (t,\omega),\quad t<0,\quad\omega\in
Q,
\end{align*}
where $f_{O}(x)\equiv Of(x)$, which proves

%
%
%

\begin{proposition}
\label{T:findprojections}Suppose $f(x)$ is a continuous function compactly
supported in $\Omega\subset Q,$ vector $\omega$ lies strictly in $Q$, and $t<0.$ Then, Radon
projections $\left(  \mathcal{R}\left[  f_{O}\right]  \right)  (t,\omega)$ of
$f_{O}(x)$ can be reconstructed from the boundary values of solution
$p(\cdot,y)$ of the problem (\ref{E:IVP}) by the formula
\begin{equation}
\left(  \mathcal{R}\left[  f_{O}\right]  \right)  (t,\omega)=\int%
\limits_{\partial Q}\left[  \sum_{j=0}^{2N-1}\sigma_{j}\left(  n(y)\cdot
\omega^{(j)}\right)  p(y\cdot\omega^{(j)}-t,y)\right]
dy,\label{E:findprojections}%
\end{equation}
where coefficients $\sigma_{j}$ and vectors $\omega^{(j)}$ have the same
values as in (\ref{E:Ustar-sum}) and (\ref{E:Ustar}).
\end{proposition}

\begin{corollary}
\label{T:limited-angles}Suppose function $f(x),$ coefficients $\sigma_{j},$
and vectors $\omega^{(j)}$ and $\omega$ are the same as in the previous
proposition. Suppose, in addition, that region  $\Omega\subset Q$ is contained
within a ball $B(0,r_{0})$ of radius $r_{0}$ and centered at $0$. Then there
is a value $R$ (given by  equation (\ref{E:finitesupport}) with
$\varepsilon=0$) such that the integration in (\ref{E:findprojections})\ can
be restricted to a finite subset of the boundary $\partial Q\cap
B(0,R(\omega))$:%
\begin{equation}
\left(  \mathcal{R}\left[  f_{O}\right]  \right)  (t,\omega)=\int%
\limits_{\partial Q\cap B(0,R(\omega))}\quad\left[  \sum_{j=0}^{2N-1}%
\sigma_{j}\left(  n(y)\cdot\omega^{(j)}\right)  p(y\cdot\omega^{(j)}%
-t,y)\right]  dy,\quad r_{0}<t<0. \label{E:findshortproj}
\end{equation}

\end{corollary}


In order to find projections for directions of $\omega$ not lying within $Q,$
we observe that the Radon transform of an odd function $Of(x)$ has many
redundancies. First, it is well known \cite{Natterer} and is easily seen from
the definition (\ref{E:Radon}) that Radon transform $\mathcal{R}h$ of a
general function $h$ is redundant in that
\begin{equation}
\left(  \mathcal{R}h\right)  (-t,-\omega)=\left(  \mathcal{R}h\right)
(t,\omega). \label{E:Rad-reddun}%
\end{equation}

Let us show that the Radon transform of an odd function $Of(x)$ has additional
redundancies. In 2D, using Proposition \ref{T:additionalsym2D}, one obtains%

\begin{align}
\left(  \mathcal{R}\left[  f_{O}\right]  \right)  \left(  t,\Upsilon_{2\beta
}^{k}\omega\right)   &  =\lim_{\varepsilon\rightarrow0}\int\limits_{\Omega
}f(x)U_{\varepsilon,\Upsilon_{2\beta}^{k}\omega}(t,x)dx=\lim_{\varepsilon
\rightarrow0}\int\limits_{\Omega}f(x)Ou_{\varepsilon,\omega}\left(
t,\Upsilon_{2\beta}^{-k}x\right)  dx\nonumber\\
&  =\lim_{\varepsilon\rightarrow0}\int\limits_{\Omega}f(x)Ou_{\varepsilon
,\omega}\left(  t,x\right)  dx=\left(  \mathcal{R}\left[  f_{O}\right]
\right)  \left(  t,\omega\right)  , \label{E:addsymproj1}%
\end{align}
and, similarly%
\begin{equation}
\left(  \mathcal{R}\left[  f_{O}\right]  \right)  \left(  t,\mathfrak{R}%
\Upsilon_{2\beta}^{k}\omega\right)  =-\left(  \mathcal{R}\left[  f_{O}\right]
\right)  \left(  t,\omega\right)  . \label{E:addsymproj2}%
\end{equation}

In 3D, with the help of Proposition \ref{T:additionalsym3D}, we observe the
following symmetries%
\begin{equation}
\left\{
\begin{array}
[c]{c}%
\left(  \mathcal{R}\left[  f_{O}\right]  \right)  \left(  t,\mathfrak{R}%
_{i}\omega\right)  =-\left(  \mathcal{R}\left[  f_{O}\right]  \right)  \left(
t,\omega\right) \\
\left(  \mathcal{R}\left[  f_{O}\right]  \right)  \left(  t,\mathfrak{R}%
_{i}\mathfrak{R}_{j}\omega\right)  =\left(  \mathcal{R}\left[  f_{O}\right]
\right)  \left(  t,\omega\right) \\
\left(  \mathcal{R}\left[  f_{O}\right]  \right)  \left(  t,\mathfrak{R}%
_{i}\mathfrak{R}_{j}\mathfrak{R}_{k}\omega\right)  =-\left(  \mathcal{R}%
\left[  f_{O}\right]  \right)  \left(  t,\omega\right)
\end{array}
\right.  \quad i,j,k\in\{1,2,3\}. \label{E:addsymproj3}%
\end{equation}
Therefore, after projections $\left(  \mathcal{R}\left[  f_{O}\right]
\right)  (t,\omega)$ of the function $f_{O}(x)$ are reconstructed using
Proposition \ref{T:findprojections} from the data $p(t,y)$ for all $\omega\in
Q$ \ and all $t<0,$ one can recover projections for other values of $\omega$
using formulas (\ref{E:addsymproj1}) and (\ref{E:addsymproj2}) (in 2D) or
formulas (\ref{E:addsymproj3}) (in 3D). Values of $\left(  \mathcal{R}\left[
f_{O}\right]  \right)  (t,\omega)$ for $t>0$ are then recovered using equation
(\ref{E:Rad-reddun}). For directions of $\omega$ parallel to the lines of
Coxeter cross (in 2D) or to coordinate planes (in 3D), Radon transform
$\left(  \mathcal{R}\left[  f_{O}\right]  \right)  (t,\omega)$ vanishes. This
follows from the fact that $f_{O}$ is odd with respect to reflections about
the lines of Coxeter cross (in 2D) or about the coordinate planes (in 3D).
Therefore, integration over hyperplanes orthogonal to these lines or planes
yields zero. Finally, since both sides of the equation
(\ref{E:findprojections}) are well-defined and continuous at $t=0$, this
formula extends to $t=0$ by continuity.

Thus, we have proven the following


\begin{theorem}
\label{T:final} Suppose $f(x)$ is a continuous function compactly supported
within a proper subset $\Omega$ of $Q$. Then, the Radon projections $\left(
\mathcal{R}\left[  f_{O}\right]  \right)  (t,\omega),$ $t\in R,$ $\omega
\in\mathbb{S}^{d-1}$, of the function $f_{O}(x)\equiv
Of(x)$ can be explicitly reconstructed from the boundary values of solution
$p(\cdot,y)$ of the problem (\ref{E:IVP})$,$ using \ formulas
(\ref{E:findprojections})-(\ref{E:addsymproj2}) (in 2D), or formulas
(\ref{E:findprojections}), (\ref{E:Rad-reddun}) and (\ref{E:addsymproj3}) (in
3D). If $\omega$ is parallel to the lines of Coxeter cross (in 2D) or to
coordinate planes (in 3D), then $\left(  \mathcal{R}\left[  f_{O}\right]
\right)  (t,\omega)=0,$ $t\in R$.
\end{theorem}

Finally, once projections $\left(  \mathcal{R}\left[  f_{O}\right]  \right)
(t,\omega)$ are found for all values of $t$ and all values of $\omega$ lying
on the unit circle (in 2D) or on the unit sphere (in 3D), function $f_{O}(x)$
can be reconstructed by application of one of the many known inversion
techniques for the Radon transform (see, for example \cite{Natterer,Kuchment}). As
mentioned above, function $f(x)$ we seek to recover is just a restriction of
$f_{O}(x)$ to $Q$.

In addition, Corollary \ref{T:limited-angles} implies that if data $p(t,y)$
are known only on a finite subset $\partial Q\cap B(0,R(\omega))$ of the
boundary of $Q$,  and $f$ is supported within a ball of radius $r_{0}$
centered at the origin, projections  $\left(  \mathcal{R}\left[  f_{O}\right]
\right)  (t,\omega)$ can be reconstructed for those directions $\omega$ that
point inside $Q$ and satisfy the inequality $R\geq2r_{0}/(1-\alpha(\omega))$,
or%
\begin{equation}
\alpha(\omega)\leq1-\frac{2r_{0}}{R}. \notag
\end{equation}
Parameter $\alpha(\omega)$ in the above inequality was defined in Section
\ref{S:repsinglay} as the upper bound on the dot products of vectors $\omega^{(j)}$
with all unit vectors lying in $\partial Q.$ Due to the symmetries in the set
of $\omega^{(j)}$, $j=0,..2N-1$, and compactness of the set $\partial
Q\cap\mathbb{S}^{d-1},$%
\begin{equation}
\alpha(\omega)=\max_{\zeta=\partial Q\cap\mathbb{S}^{d-1}}(\omega\cdot\zeta). \notag
\end{equation}
Therefore, for fixed $R$ and $r_{0}$, and for
$t>0,$ projections $\left(  \mathcal{R}\left[  f_{O}\right]  \right)
(t,\omega)$ can be reconstructed in the interval of values of $\omega$
satisfying the inequality
\begin{equation}
\max_{\zeta=\partial Q\cap\mathbb{S}^{d-1}}(\omega\cdot\zeta)\leq1-\frac{2r_{0}%
}{R}. \notag
\end{equation}
In 3D, this inequality describes a ''triangle-like" connected region
$T(R,r_{0})$ that is a subset of the set of all unit vectors with positive
coordinates. In 2D, if vector $\omega$ is represented as $(\cos\gamma
,\sin\gamma)$ with $\gamma\in\lbrack0,\beta],$ (recall that $\beta=\pi/N,$
$N=2,3,4,...),$ the above inequality defines the interval $I(R,r_{0}%
)\equiv\lbrack\gamma_{0},\beta-\gamma_{0}]$, where $\gamma_{0}=\arccos\left(
1-2r_{0}/R\right)  .$ By using the symmetries expressed by equations
(\ref{E:Rad-reddun})-(\ref{E:addsymproj3}), one then finds values of $\left(
\mathcal{R}\left[  f_{O}\right]  \right)  (t,\omega)$ that corresponds to
rotations and/or reflections of $I(R,r_{0})$ or $T(R,r_{0}),$ and to
nonnegative values of $t.$

\begin{theorem}
\label{T:veryfinal}
Suppose $f(x)$ is a continuous function compactly supported within a proper
subset $\Omega$ of $Q\cap B(0,r_{0})$, and boundary values of solution
$p(\cdot,y)$ of the problem (\ref{E:IVP}) are known for all $y\in\partial
Q^{\ast}\equiv\partial Q\cap B(0,R),$where $r_{0}$ and $R$ are arbitrary (but
with $R>2r_{0})$.  Then, in 3D the Radon projections $\left(  \mathcal{R}%
\left[  f_{O}\right]  \right)  (t,\omega)$ of the function $f_{O}(x)\equiv
Of(x)$ can be explicitly reconstructed from $p(\cdot,y),$ $y\in\partial
Q^{\ast},$ for all values of $\omega=(\omega_{1},\omega_{2},\omega_{3})$ such
that $\max\limits_{j=1,2,3}|\omega_{j}|\leq1-\frac{2r_{0}}{R},$ and all $t$
with $|t|\leq r_{0},$ using equations (\ref{E:findshortproj}),
(\ref{E:Rad-reddun}) and (\ref{E:addsymproj3}). In 2D, the Radon projections
$\left(  \mathcal{R}\left[  f_{O}\right]  \right)  (t,\omega(y))$ of the
function $f_{O}(x)\equiv Of(x)$ can be explicitly reconstructed from
$p(\cdot,y),$ $y\in\partial Q^{\ast},$ for all $\gamma\in\bigcup
\limits_{k=0}^{2N-1}[\gamma_{0}+\beta_{k},\beta-\gamma_{0}+\beta k],$ and all
$t$ with $|t|\leq r_{0},$ using formulas (\ref{E:findshortproj}) and
(\ref{E:Rad-reddun})-(\ref{E:addsymproj2}). (In all dimensions  $\left(
\mathcal{R}\left[  f_{O}\right]  \right)  (t,\omega)=0$ for $t>r_{0}.)$
\end{theorem}

\begin{figure}[t]
\begin{center}
\subfigure[Phantom $f(x)$]{  \phantom{aaaa}
\includegraphics[width=1.6in,height=1.6in]{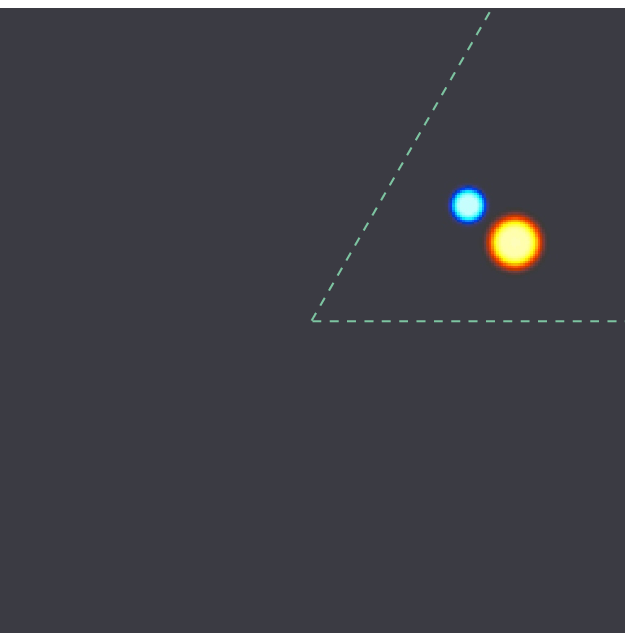} \phantom{aaaa}}
\subfigure[Corresponding $f_O(x)$, $N=3$]{        \phantom{aaaa}
\includegraphics[width=1.6in,height=1.6in] {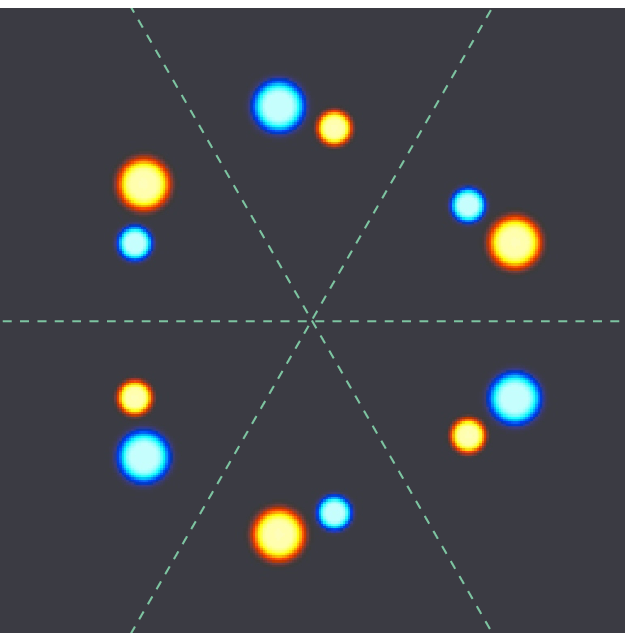}
\includegraphics[width=0.32in,height=1.6in]{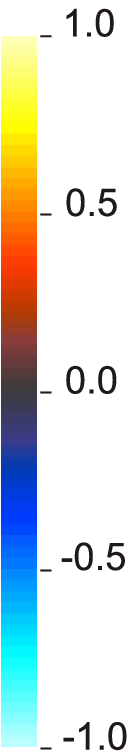}} \\
\subfigure[Projections $\left( \mathcal{R} f_{O} \right) (t,\omega(\theta))$]{
\includegraphics[width=2.35in,height=1.6in]{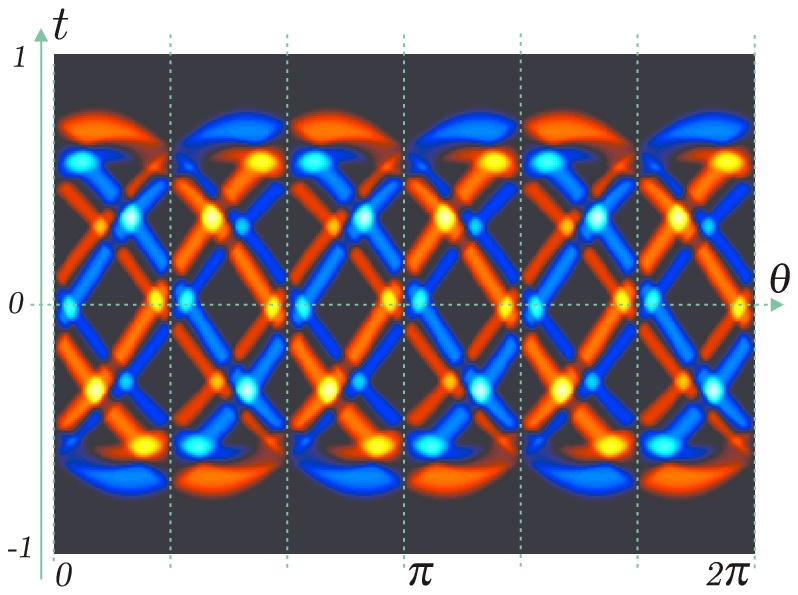}
\includegraphics[width=0.32in,height=1.6in]{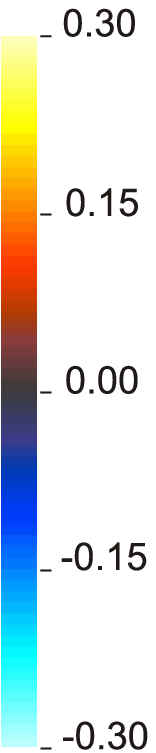}}
\subfigure[Error in the reconstructed projections]{ \phantom{aa}
\includegraphics[width=2.35in,height=1.6in]{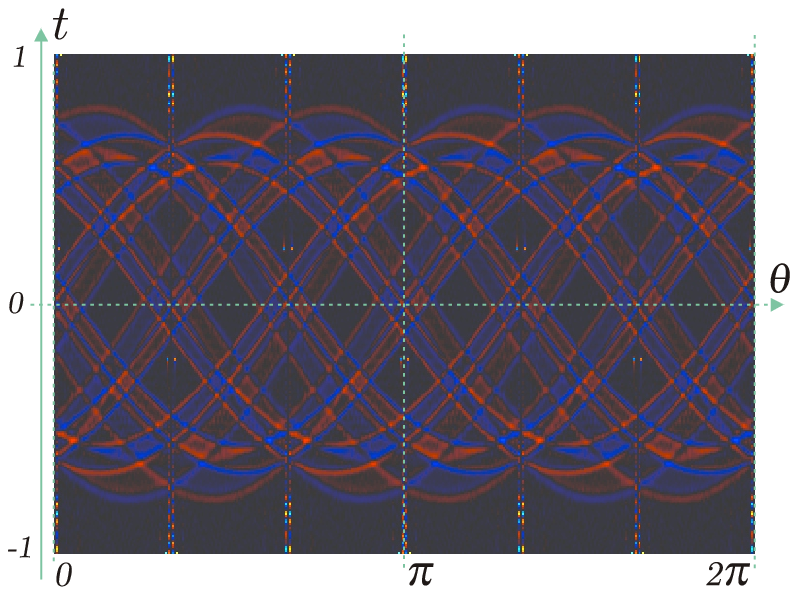}
\includegraphics[width=0.32in,height=1.6in]{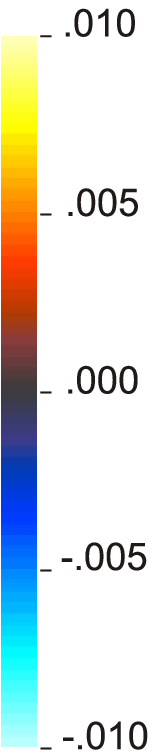}}
\subfigure[Reconstruction from $p(t,y)$ with $|y|\le 300$]{ 
\includegraphics[width=2.35in,height=1.6in]{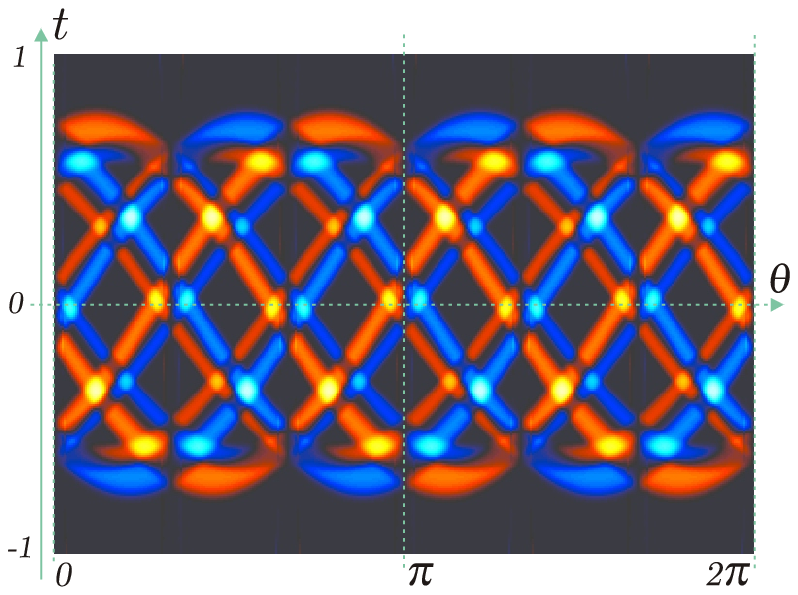}
\includegraphics[width=0.32in,height=1.6in]{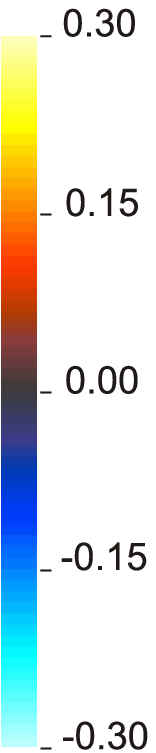}}
\subfigure[Reconstruction from noisy data, $|y|\le 300$]{ \phantom{a}
\includegraphics[width=2.35in,height=1.6in]{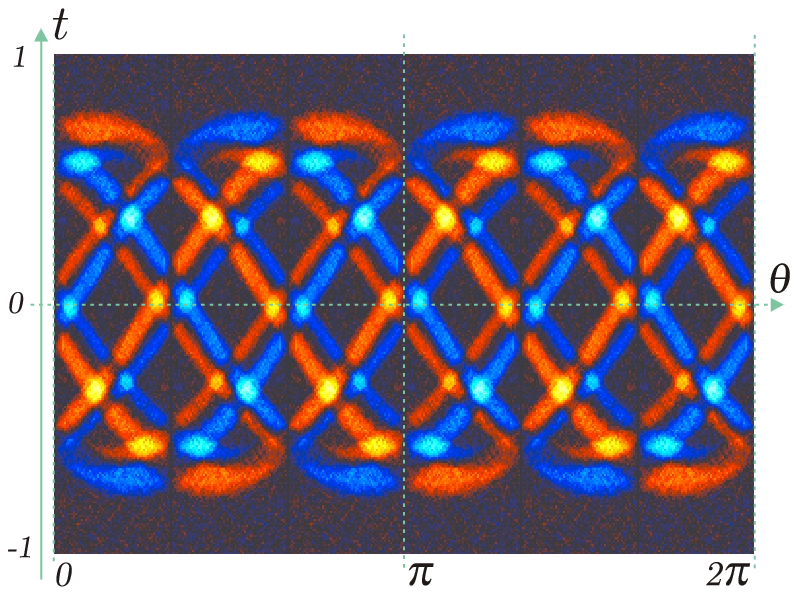}
\includegraphics[width=0.32in,height=1.6in]{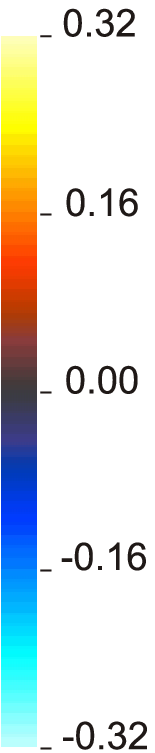}}

\end{center}
\caption{Example of reconstructing projections from $p(t,y)$ in 2D, $N=3$}
\label{F:example2D}
\end{figure}

\section{Numerical simulations}

In this section our formulas for reconstructing Radon projections are tested
in two numerical simulations, one in 2D and the other one in 3D. The goal of
these numerical experiments is to illustrate the exactness of the formulas,
rather than to model realistic thermo- or photo- acoustic reconstruction. To
this end, we did not simulate measurement noise, and did not study the effects
of truncation of the acquisition surfaces. Moreover, when modeling the forward
problem, an effort was made to sample solution of the wave equation on a grid
covering the whole infinite boundary~$\partial Q$. Clearly, this cannot be
done numerically on a uniform grid. We, thus, used grids that are uniform on
parts of the boundary relatively close to the origin (e.g. $|y|$ is less than
several hundreds), and that are getting progressively coarser for larger
values of $|y|.$

As stated by Corollary \ref{T:limited-angles}, for a particular direction of
$\omega$ integration is done only over a finite subset of $\partial Q.$ The
remote parts of $\partial Q$ are only needed for computing projections with
$\omega$ near parallel to one of the parts of the boundary. Since remote parts
of $\partial Q$ were sampled very coarsely, values of reconstructed projections
corresponding to such directions of $\omega$ were less accurate.

\subsection{A corner-like domain in 2D}

Consider a problem of reconstructing Radon projections from solution of the
wave equation known on $\partial Q$ consisting of two rays with angle
$\beta=\pi/3$ between them. \ This geometry corresponds to a subset of Coxeter
cross with $N$ equal to $3.$ As a phantom $f(x)$, we use a linear combination
$f(x)=h_{1}(x)-h_{2}(x)$ of functions $h_{j},$ $j=1,2,$ in the form%
\begin{equation}
h_{j}=h\left(  \frac{|x-x_{j}|}{r_{j}}\right)  , \label{E:phantom}%
\end{equation}
where $h(s)$ is a smooth function with a maximum equal to 1, at $s=0$, and
vanishing with several of its derivatives at $s=1.$ This phantom is plotted in
figure \ref{F:example2D}(a). The odd part $f_{O}(x)$ of the function $f(x)$ is
defined by formula (\ref{E:odd-oper}); it is shown in figure \ref{F:example2D}%
(b).

Values of $p(t,y)$\ corresponding to the chosen phantom $f(x)$ were computed
on $\partial Q$ by first computing circular averages $M(y,r)$ defined by the
equation
\begin{equation}
M(y,r)\equiv\int\limits_{\mathbb{S}^{1}}\frac{f(y+r\alpha)}{\sqrt
{t^{2}-|r|^{2}}}d\alpha,
\notag \end{equation}
and by exploiting a well known connection between $p(t,y)$ and $M(y,r):$%
\begin{align*}
p(t,y)  &  =\frac{\partial}{\partial t}\int\limits_{\Omega}f(x)G^{+}%
(t,y-x)dx=\frac{\partial}{\partial t}\int\limits_{\Omega}f(y+z)G^{+}(t,z)dz\\
&  =\frac{1}{2\pi}\frac{\partial}{\partial t}\int\limits_{t}^{\infty}%
\int\limits_{\mathbb{S}^{1}}\frac{f(y+r\alpha)}{\sqrt{t^{2}-|r|^{2}}}%
d\alpha\ rdr=\frac{1}{2\pi}\frac{\partial}{\partial t}\int\limits_{t}^{\infty
}\frac{M(y,r)}{\sqrt{t^{2}-|r|^{2}}}\ rdr.
\end{align*}
The Abel transform of $M(y,r)$ and its time derivative
were computed numerically.

\begin{figure}[t]
\begin{center}
\subfigure[Phantom, $z=0.3$]{
\includegraphics[width=1.6in,height=1.6in]{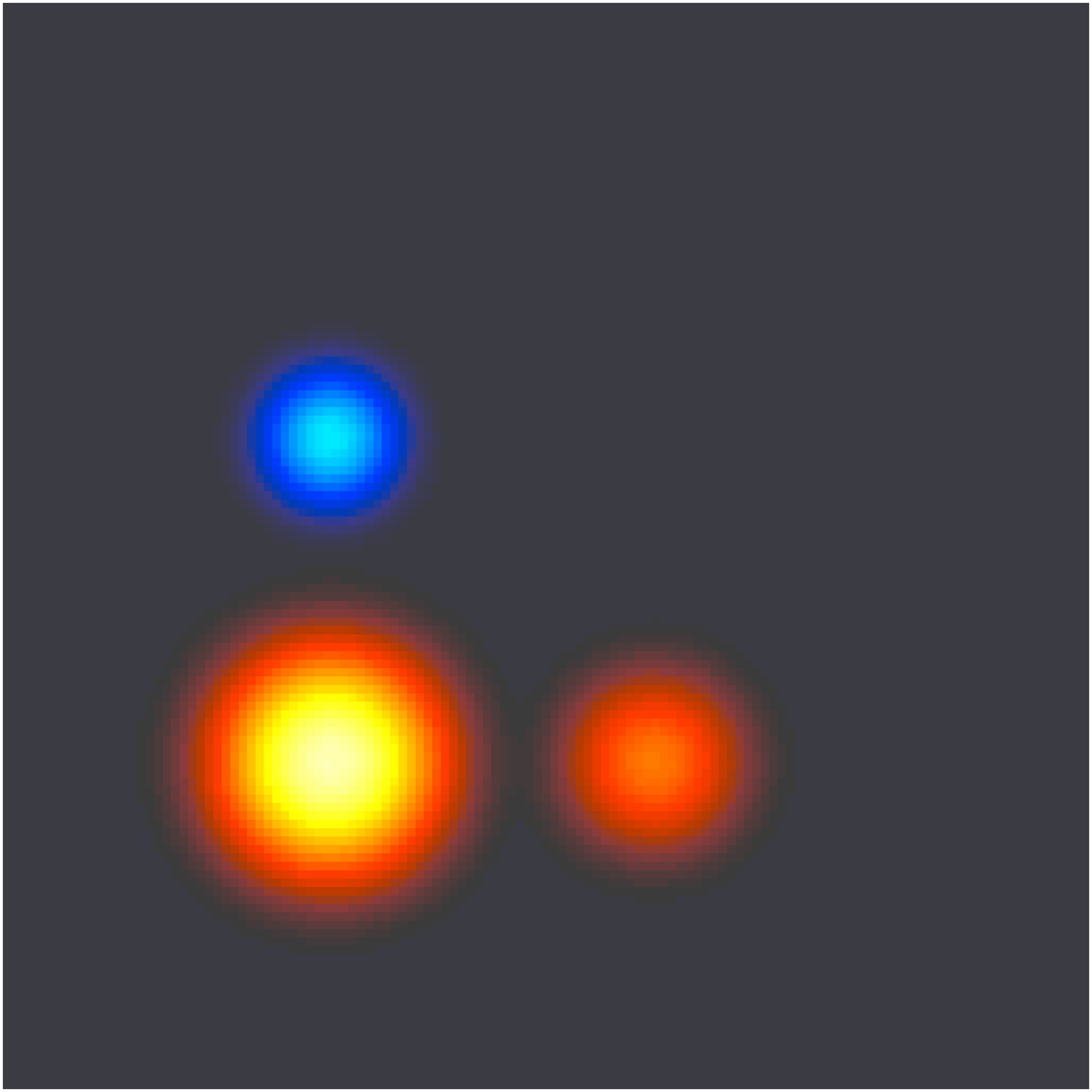} }
\subfigure[Phantom, $z=0.6$]{
\includegraphics[width=1.6in,height=1.6in] {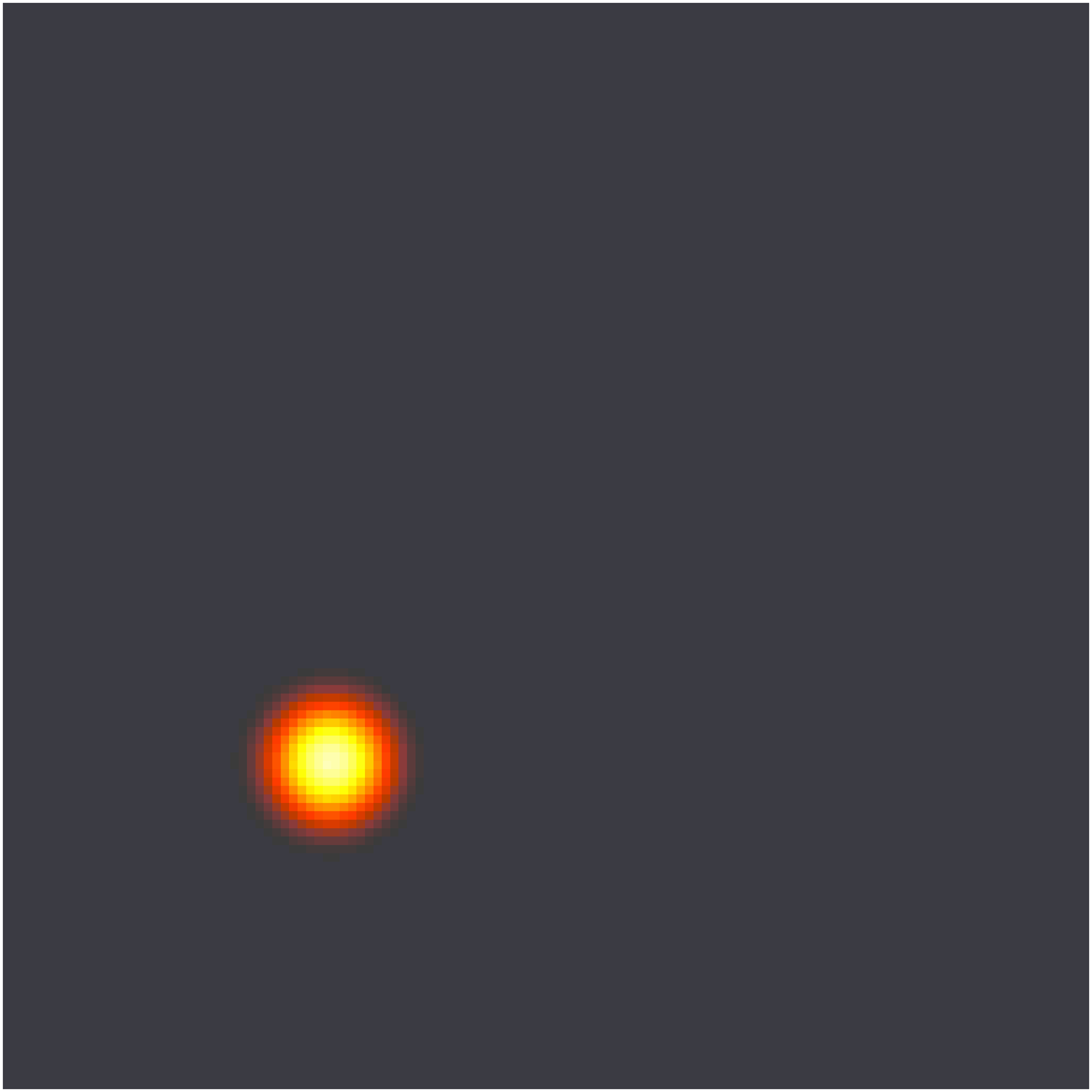}
\includegraphics[width=0.32in,height=1.6in]{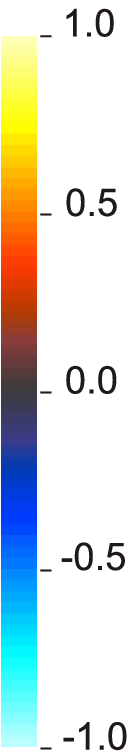} }
\subfigure[Projections $\left( \mathcal{R} f_{O} \right) (t,\omega(\theta,\varphi_0))$]{
\includegraphics[width=2.20in,height=1.6in]{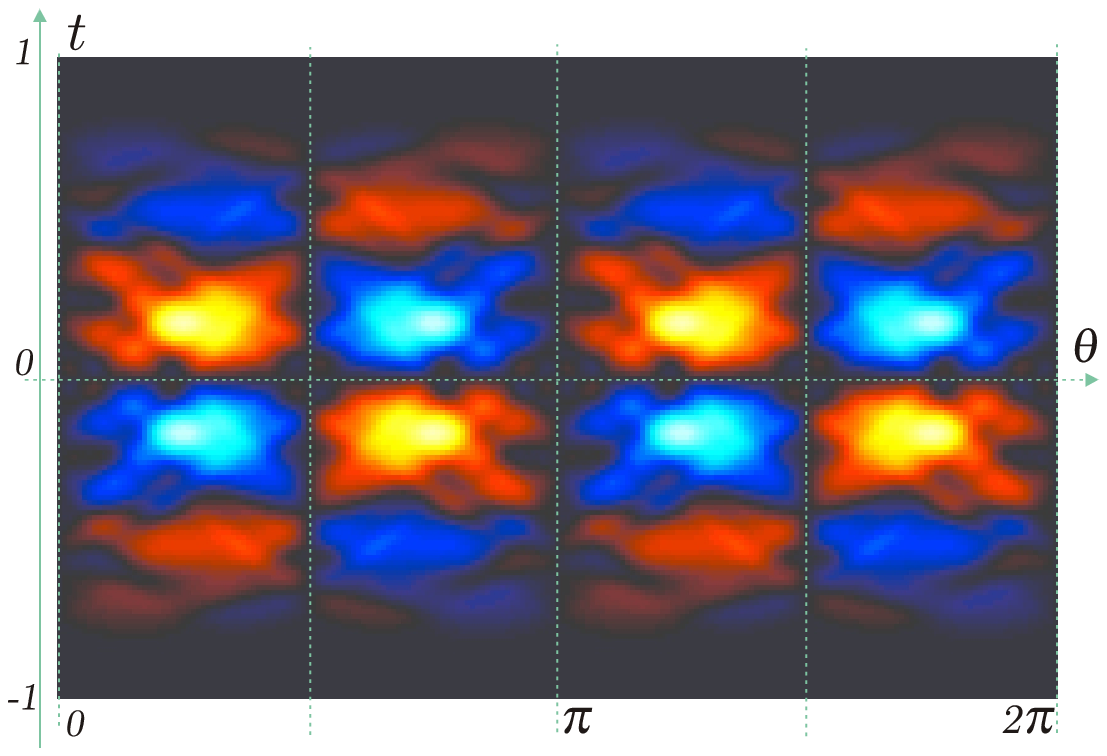}}\\
\subfigure[Projections on $(-1,0)\times (0,\pi/2)$, $\varphi=\varphi_0$]{ \phantom{aaa}
\includegraphics[width=1.9in,height=2.2in]{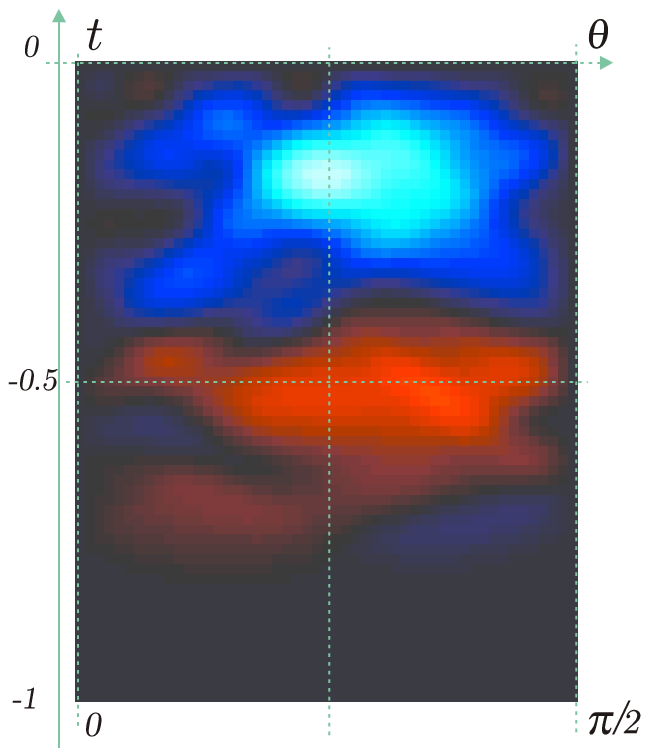}
\includegraphics[width=0.4in,height=2.2in]{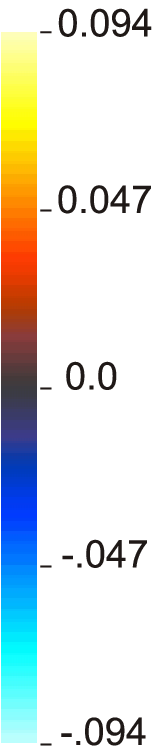} \phantom{aaa} }
\subfigure[Recon. error on $(-1,0)\times(0,\pi/2)$, $\varphi=\varphi_0$]{  \phantom{aa}
\includegraphics[width=1.9in,height=2.20in]{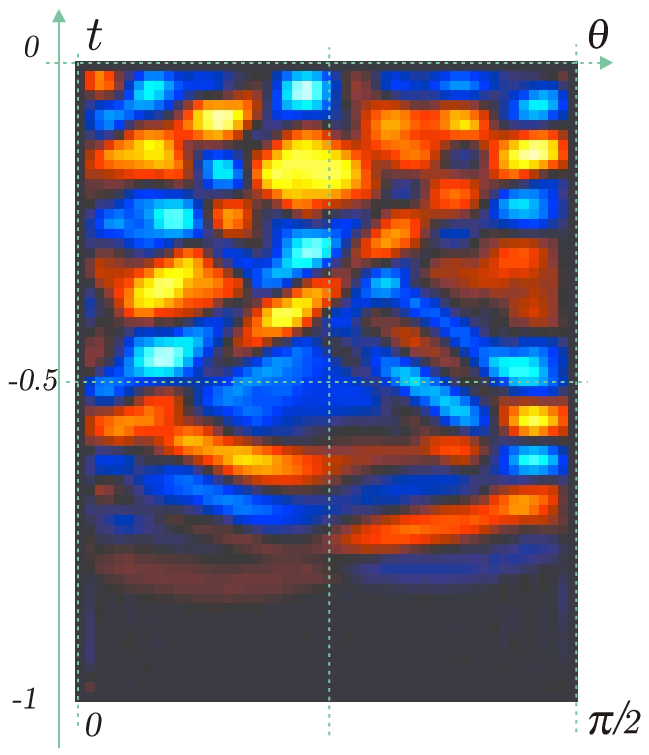}
\includegraphics[width=0.4in,height=2.2in] {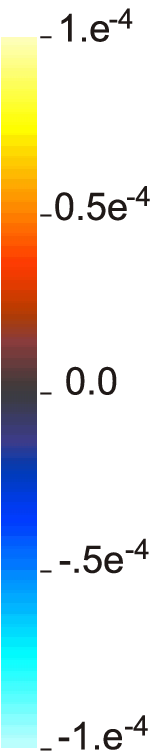}  \phantom{aaaa}}
\end{center}
\caption{Reconstructing projections $\left( \mathcal{R} f_{O} \right) (t,\omega(\theta,\varphi))$ from values of $p(t,y)$ known  on $\partial Q$
in 3D}
\label{F:example3D}
\end{figure}


For further comparison, the Radon projections $\left(  \mathcal{R}\left[
f_{O}\right]  \right)  (t,\omega(\theta))$ of $f_{O}(x)$ were computed first
using the exact $f_{O}(x).$The result is shown in figure~\ref{F:example2D}(c).
One can clearly see in this figure symmetries described by equations
(\ref{E:Rad-reddun})-(\ref{E:addsymproj2}). Next, we reconstructed the values
of projections from $p(t,y)$ using formulas (\ref{E:findprojections}%
)-(\ref{E:addsymproj2}). On a gray scale (or color) plot the result is
indistinguishable from the one in figure~\ref{F:example2D}(c). \ The plot of
the difference between the reconstructed and known projections is shown in
figure~\ref{F:example2D}(d). One can notice that the error for values of
$\theta$ close to $\frac{\pi k}{3},$ $k\in\mathbb{Z}$, is larger than for the
rest of these values. An explanation was presented right before the start of
this section. In an $L_{\infty}$-norm, the relative error over the whole range
of values $t$ and $\theta$ is about 3\%. If values of $\theta$ close to
$\frac{\pi k}{3}$ are excluded from the comparison, this error drops to about
1\%. At least a part of this error is due to the error in computing $p(t,y)$
by a three-step process involving finding circular averages, computing the
Abel transform, and subsequent numerical differentiation.

In figure~\ref{F:example2D}(e) we present an image reconstructed from data
$p(t,y)$ known only for values $|y| \le 300$. According to Theorem~\ref{T:veryfinal}
applied with values $r_0=1$, $R=300$, our formulas yield exact reconstruction only
for values of $\omega(\gamma)$ with $ \gamma \in\bigcup
\limits_{k=0}^{5}[6.62 ^\circ + k 60 ^\circ, (k+1) 60 ^\circ  - 6.62 ^\circ]$.
If one examines the image in part~(e) only within this set of
values, the error is similar to that shown in figure~\ref{F:example2D}(d).
For other values of $\gamma$ the error is larger, and can be noticed
in a high-resolution version of the figure.

Figure~\ref{F:example2D}(f) demonstrates reconstruction from
noisy (and truncated in $y$) data. A $100\%$ noise (in $L_2$ norm) was added
to values $p(t,y)$ at each point $y$ with $|y| \le 300$, on the time intervals
of length 2 starting at the time of arrival of the acoustic waves to point $y$.
The effect of noise is noticeable in the figure; the relative $L_\infty$ norm
of the error is about $25\%$. This error, however, is quite small taking
into account the high level of noise in the data --- as one would expect
in view of the highly stable nature of our reconstruction formulas.

\subsection{Boundary of an octet in 3D}

In the second numerical experiment we reconstructed Radon projections of a
function from solution of the wave equation known on a boundary $\partial Q$
of the first octet in 3D. As a phantom $f(x)$ defined within the intersection
of the unit ball and the first octet, we used a linear combination of four
functions in the form (\ref{E:phantom}). The centers $x_{j}$ of these
functions lied either in the plane $z=0.3$ or plane $z=0.6.$ The
cross-sections of the phantom by these two planes are shown in
figure~\ref{F:example3D}, parts (a) and (b). Projections $\left(
\mathcal{R}\left[  f_{O}\right]  \right)  (t,\omega)$ of function $f_{O}(x)$
with $\omega$ parametrized in spherical coordinates (i.e., $\omega
=(\sin\varphi\cos\theta,\sin\varphi\sin\theta,\cos\varphi))$ are shown in
figure~\ref{F:example3D}(c) for a fixed value of $\varphi=\varphi
_{0}\thickapprox53^{\circ}$. Symmetries expressed by equations
(\ref{E:Rad-reddun}) and (\ref{E:addsymproj3}) can be clearly seen in this
figure. A magnified image of projections restricted to the region
$(t,\theta)\in\lbrack-1,0]\times\lbrack0,\pi/2]$ with $\varphi=\varphi_{0}$ is
shown in figure~\ref{F:example3D}(d).

Modeling the forward problem (i.e., computing values of $p(t,y)$ on $\partial
Q)$ is actually simpler in 3D than in 2D. For radial functions in the form
(\ref{E:phantom}) there exists an explicit exact formula \cite{Haltm-series}
yielding values of $p(t,y).$ We precomputed these values on polar grids
defined on each of three infinite faces constituting $\partial Q.$ The
discretization step in the radial direction was uniform for relatively small
values of the radius but became increasingly coarser for larger values.


Numerical reconstruction of projections from $p(t,y)$ performed using formulas
(\ref{E:findprojections}), (\ref{E:Rad-reddun}) and (\ref{E:addsymproj3}) in
the region $(t,\theta)\in\lbrack-1,0]\times\lbrack0,\pi/2]$ yields data that
are very close to the exact ones for most values of $\varphi.$ The results
obtained for $\varphi=\varphi_{0}$ are indistinguishable from the exact values
on a gray-scale or color plot. The error for $\varphi=\varphi_{0}$ is plotted
in figure~\ref{F:example3D}(e); the relative error in $L_{\infty}$ norm is
about 0.1\%. Similar error is observed for wide range of values of $\varphi$
except those close to $0$ or $\pi/2$. We explain a better accuracy obtained in
3D simulation by a better accuracy in modeling $p(t,y)$: in the 3D case these
values are given by a theoretically exact formula.

\section{Concluding remarks and acknowledgement}

We presented explicit formulas that reconstruct Radon projections of $f_{O}$
from the values of the
wave equation known on a boundary of an
angular domain with $\pi/N$, $N=2,3,4,..$ (in 2D) or on the boundary of octet
$Q$ (in 3D). Here are some additional thoughts:

\begin{itemize}
\item Since the solution of the wave equation in the case of constant speed of
sound can be expressed through the spherical/circular means (and vice versa),
finding Radon projections from these means is straightforward.

\item Explicit relations between the Radon transform and the
spherical/circular means were found in \cite{Beltukov} in the case when the
measurement surface is a hyperplane. The present results may be considered a
generalization of that work, although the case of a hyperplane in 2D or 3D is
not presented in our paper.

\item General approach of this paper, as well as that of~\cite{Beltukov}, is
somewhat similar to the approach of Popov and Sushko~\cite{PS1,PS2} who
proposed a formula for finding Radon projections of a function from
optoacoustic measurements. Their formula, however, yields
a parametrix of the problem rather than the exact solution.

\item Theorem~\ref{T:final} can be easily generalized to the case when the
acquisition surface (in 3D) is the boundary of an infinite region bounded by
three planes, with two of these planes intersecting at the angle $\pi/N$,
$N=3,4,5...$ and the third plane perpendicular to the first two.
\end{itemize}

\textbf{Acknowledgement:} The author's research is partially supported by the
NSF, award NSF/DMS-1211521.

\end{document}